\documentclass[a4paper,reqno,
11pt]{amsart}
\usepackage{
amssymb,
amsmath,
amsthm,
eucal,
empheq,
cases,
dsfont,
multicol,
mathrsfs,
tikz,
graphicx,
hyperref,
esvect
}

\usepackage{lipsum}
\makeatletter
\renewcommand*{\eqref}[1]{%
\hyperref[{#1}]{\textup{\tagform@{\!\!\ref*{#1}}}}%
}

\makeatletter
 
  \@addtoreset{equation}{section}
 \makeatother
\makeatletter

\newcommand{\opnorm}{\@ifstar\@opnorms\@opnorm}
\newcommand{\@opnorms}[1]{%
  \left|\mkern-1.5mu\left|\mkern-1.5mu\left|
   #1
  \right|\mkern-1.5mu\right|\mkern-1.5mu\right|
}
\newcommand{\@opnorm}[2][]{%
  \mathopen{#1|\mkern-1.5mu#1|\mkern-1.5mu#1|}
  #2
  \mathclose{#1|\mkern-1.5mu#1|\mkern-1.5mu#1|}
}

\makeatother\theoremstyle{plain}
\newtheorem{theorem}{Theorem}[section]
\newtheorem{lemma}[theorem]{Lemma}
\newtheorem{proposition}[theorem]{Proposition}
\newtheorem{corollary}[theorem]{Corollary}
\theoremstyle{definition}

\newtheorem{remark}[theorem]{Remark}

\newcommand{\bvec}[1]{\mbox{\boldmath $#1$}}

\def\Re{\mathop{\mathrm{Re}}\nolimits}
\def\Im{\mathop{\mathrm{Im}}\nolimits}

\def\diag{\mathop{\mathrm{diag}}\nolimits}

\def\R{{\mathbb{R}}}

\def\C{{\mathbb{C}}}

\def\F{{\mathcal{F}}}
\def\H{{\mathcal{H}}}

\def\L{{\mathcal{L}}}
\def\X{{\mathcal{X}}}

\def\<{{\langle}}
\def\>{{\rangle}}

\def\ep{{\varepsilon}}

\title[Large data modified wave operators for 1D cubic NLS]{Modified wave operators for the defocusing cubic nonlinear Schr\"odinger equation in one space dimension with large scattering data}

\author{Masaki Kawamoto}
\address[M. Kawamoto]{Research Institute for Interdisciplinary Science, Okayama University, 3-1-1, Tsushimanaka, Kita-ku, Okayama City, Okayama, 700-8530, Japan}
\email{kawamoto.masaki@okayama-u.ac.jp}
\author{Haruya Mizutani}
\address[H. Mizutani]{Department of Mathematics, Graduate School of Science, The University of Osaka, Toyonaka, Osaka 560-0043, Japan}
\email{haruya@math.sci.osaka-u.ac.jp}

\makeatletter
\@namedef{subjclassname@2020}{%
	\textup{2020} Mathematics Subject Classification}
\makeatother


\begin{document}

%
\begin{abstract}
In the present paper, we construct modified wave operators for the defocusing cubic nonlinear Schr\"odinger equation (NLS) in one space dimension without size restriction on scattering data. In the proof, we introduce a new formulation of the problem based on the linearization of the NLS around a prescribed asymptotic profile. For the linearized equation which is a system of Schr\"odinger equations with non-symmetric, time-dependent long-range potentials, we show a modified energy identity, as well as an associated energy estimate, which allow us to apply a simple energy method to construct the modified wave operators. As a byproduct, we also obtain in the focusing case an improved explicit upper bound for the size of scattering data to ensure the existence of modified wave operators. Our argument relies neither on the complete integrability nor on the framework of analytic function spaces, and also works for short-range perturbations of the cubic nonlinearity. 
\end{abstract}

\maketitle

\section{Introduction}
\subsection{Introduction}
\label{introduction}

In this paper we are interested in scattering theory for the following nonlinear Schr\"odinger equation (NLS) in one space dimension: 
\begin{align}
\label{NLS}
i\partial_t u-H_0u=\lambda_1 |u|^2u+\lambda_2|u|^{2\sigma}u,\quad x\in \R,\quad t\in \R,
\end{align}
where $u=u(t,x)$ is a $\C$-valued unknown function, $\lambda_1,\lambda_2\in \R$, $1<\sigma<2$ and
$$
H_0=-\frac{1}{2}\frac{d^2}{dx^2}. 
$$
The cubic nonlinearity is critical in the context of scattering theory in the sense that if $\lambda_1\neq0$, then no non-trivial solution to \eqref{NLS} scatters to a solution to the free Schr\"odinger equation regardless of the defocusing case $\lambda_1>0$ or the focusing case $\lambda_1<0$  (see \cite{Strauss,Barab,Cazenave}). Instead, appropriate modifications of asymptotic profiles depending on the cubic nonlinearity must be taken into account to establish the asymptotic behavior of the solutions even for small solutions.

The main result in this paper is the modified scattering for the final state problem (FSP) and existence of modified wave operators for {\it arbitrarily large scattering data} provided $\lambda_1>0$, {\it i.e.}, the cubic nonlinearity is defocusing. Specifically, we define the asymptotic profiles $u_{\mathrm p,\pm}$ by
\begin{align}
\label{u_p}
u_{\mathrm{p},\pm}(t,x)=[\mathcal M(t) \mathcal D(t) w_{\mathrm{p},\pm}](t,x)=(it)^{-1/2}e^{i|x|^2/(2t)}e^{\mp i\lambda_1 |\widehat{u_\pm}(x/t)|^2\log |t|}\widehat{u_\pm}(x/t),
\end{align}
where $u_\pm$ are given scattering data (also called scattering states, or final data) and
\begin{align*}
\widehat f(\xi)&=\F f(\xi)=\frac{1}{\sqrt{2\pi}}\int_{\R} e^{-ix\xi}f(x)dx,\\
\mathcal M(t)f(x)&=e^{i|x|^2/(2t)}f(x),\\
 \mathcal D(t)f(x)&=(it)^{-1/2}f(x/t),\\
\nonumber
w_{\mathrm{p},\pm}(t,x)&=e^{\mp i\lambda_1 |\widehat{u_\pm}(x)|^2\log |t|}\widehat{u_\pm}(x). 
\end{align*}
Recall that the free propagator $e^{-itH_0}$ satisfies the Dollard decomposition
\begin{align}
\label{Dollard_1}
e^{-itH_0}=\mathcal M(t)\mathcal D(t)\mathcal F\mathcal M(t). 
\end{align}
Since $e^{i|x|^2/(2t)}\to1$ as $t\to \infty$ for all $x\in \R$, we know
\begin{align*}
e^{-itH_0}u_\pm =\mathcal M(t)\mathcal D(t)\widehat{u_\pm}+\mathcal M(t)\mathcal D(t)\mathcal F\left(\mathcal M(t)-I\right)u_\pm=\mathcal M(t)\mathcal D(t)\widehat{u_\pm}+o(1)
\end{align*}
in $L^2(\R)$ as $t\to \pm\infty$. The asymptotic profile $u_{\mathrm p,\pm}$ thus has the additional phase correction term $e^{\mp i\lambda_1 |\widehat{u_\pm}(x/t)|^2\log |t|}$ compared with this leading term $\mathcal M(t)\mathcal D(t)\widehat{u_\pm}$ of the free solution $e^{-itH_0}u_\pm$. 

Then, by the modified scattering for the FSP, we mean that for any scattering datum $u_+$ (resp. $u_-$), there exists a unique global solution $u$ to \eqref{NLS} which scatters to the prescribed asymptotic profile $u_{\mathrm p,+}$ (resp. $u_{\mathrm p,-}$) in the sense that 
\begin{align}
\label{intro_1}
\|u(t)-u_{\mathrm p.+}(t)\|_X\to 0\quad(\text{resp. }\|u(t)-u_{\mathrm p.-}(t)\|_X\to 0)
\end{align}
as $t\to \infty$ (resp. $t\to -\infty$) in a suitable function space $X$. This statement particularly ensures the existence of the modified wave operators $$
W_\pm:u_\pm\mapsto u(0), 
$$
which is one of main steps to construct the modified scattering operator 
$S:u_-\mapsto u_+$. The (modified) scattering operator is an important object in scattering theory to describe the correspondence between the future and past asymptotic behaviors of the solutions to \eqref{NLS}.   


The modified scattering has been extensively studied for both the Cauchy problem (CP) and FSP of \eqref{NLS}, or more generally, of the following NLS
\begin{align}
\label{NLS_1}
i\partial_t u+\frac12\Delta u=\lambda_1|u|^{2/d}u+\lambda_2 |u|^{2\sigma}u,\quad x\in \R^d,\quad t\ge0,\quad \lambda_1,\lambda_2\in \R,\quad \sigma>1/d. 
\end{align}
The modified scattering for the FSP and the existence of wave operators were first established by Ozawa's seminal paper \cite{Ozawa_1991} in the one-dimensional cubic case, and then extended by \cite{Ginibre_Ozawa_1993} to the two and three dimensional cases. The condition on the scattering data $u_\pm$, as well as the topology $X$ of the convergence \eqref{intro_1} were later improved by \cite{Hayashi_Naumkin_2006}. Precisely, it was shown in \cite{Hayashi_Naumkin_2006} that, for $1\le d\le3$, $\lambda_1\in \R$, $\lambda_2=0$, $d/2<\alpha<\min\{d,2,1+2/d\}$, $d/2<\beta<\alpha$ and sufficiently small $\ep>0$, the modified scattering for the FSP in $F H^\beta(\R)$ holds for all $u_\pm\in \mathcal FH^\alpha(\R) $ satisfying $\|\widehat{u_\pm}\|_{L^\infty}<\ep$. In particular, the modified wave operators
$$
W_\pm:\{f\in \mathcal FH^{\alpha}(\R^d)\ |\ \|\widehat f\|_{L^\infty}<\ep\}\ni u_\pm\mapsto u(0)\in \mathcal FH^\beta(\R^d)
$$ 
are well defined. The modified scattering for the CP of \eqref{NLS_1} was established by \cite{Hayashi_Naumkin_1998,Hayashi_Naumkin_2006} for $1\le d\le 3$. In \cite{KaPu,Lindblad_Soffer_2006,IfTa}, the authors provided alternative methods to establish the modified scattering for the CP in one space dimension $d=1$. We also refer to \cite{Carles_2001,Hayashi_Naumkin_2006,Carles_2024} for the construction and its properties of the modified scattering operator. 


The aforementioned papers, except for \cite{Carles_2001}, have addressed only the case with sufficiently small data, and worked in a framework based on standard weighted $L^2$ or weighted Sobolev spaces. In \cite{Carles_2001}, an upper bound of $\|\widehat{u_\pm}\|_{L^\infty}$ to ensure the modified scattering for the FSP of \eqref{NLS} was obtained. There are also several results on the large data problem based on a special feature of the equation or for suitable well-designed given data. The large data modified scattering was established by  \cite{Deift_Zhou_2003} for the CP of \eqref{NLS} in the defocusing cubic case $\lambda_1>0$ and $\lambda_2=0$ via the complete integrability of \eqref{NLS} (with $\lambda_2=0$) and inverse scattering theory (see also \cite{Deift_Zhou_2002} for the case with sufficiently small $\lambda_2>0$), and by  \cite{Ginibre_Velo_2001} for the FSP of \eqref{NLS} with $\lambda_1\in \R$ and $\lambda_2=0$ using the framework of a suitable analytic function space as the energy space. The authors of \cite{Cazenave_Naumkin_2018} utilized the non-vanishing condition $\inf_{x\in \R^d}(\<x\>^N|u_0|)>0$ with some large $N>0$ for the initial data $u_0$ to establish the modified scattering for the CP of \eqref{NLS_1} with any $d\ge1$, $\lambda_1\in \R$ and $\lambda_2=0$, where they considered arbitrarily large, but highly oscillating initial data of the form $e^{ib|x|^2}u_0$ with  large $b$.

In summary, although the small data case has been relatively well understood, the literature of the large data modified scattering for NLS is much more sparse. In particular, to the best of our knowledge, there seems to be no previous result on the large data modified scattering for the FSP of \eqref{NLS} in the framework of non-analytic function spaces, which we prove in this paper. We hope that the method of this paper will serve as a starting point for the analysis of the modified scattering with large data for more general non-integrable nonlinear dispersive equations (see Remark \ref{remark_future} below for some future topics).

\begin{remark}
It should be mentioned that the modified scattering has been also extensively studied for the long-range nonlinear Hartree equations in space dimensions $d\ge2$ of the form 
\begin{align}
\label{Hartree}
i\partial_t u+\frac12\Delta u=\lambda (|x|^{-\sigma}*|u|^2)u,\quad x\in \R^d,\quad t\ge0,\quad 0<\sigma\le1
\end{align}
(see e.g. \cite{Ginibre_Ozawa_1993,Ginibre_Velo_2000_1,Ginibre_Velo_2000_2,Ginibre_Velo_2001,Ginibre_Velo_2014,Ginibre_Velo_2015,Hayashi_Naumkin_2001,Nakanishi_CPAA,Nakanishi_AHP}). In particular, the modified scattering for the FSP was established by  Ginibre--Velo \cite{Ginibre_Velo_2000_1,Ginibre_Velo_2000_2,Ginibre_Velo_2001,Ginibre_Velo_2014,Ginibre_Velo_2015} and Nakanishi \cite{Nakanishi_CPAA,Nakanishi_AHP} without size restriction. However, it is unclear  whether their methods apply to the NLS with power nonlinearities, as the smoothing property of the convolution plays an essential role (see Section \ref{section_Hartree} for more details). 
\end{remark}

\subsection{Main result}
We shall deal with the modified scattering for the positive time direction $t\to \infty$ only, since the argument for the negative time is analogous thanks to the time reversal symmetry of \eqref{NLS}. In what follows, we denote for simplicity
$$
\varphi=\widehat{u_+},\quad u_{\mathrm p}=u_{\mathrm p,+},\quad w_{\mathrm p}=w_{\mathrm p,+}. 
$$
 For a technical reason (see Remark \ref{remark_theorem_1} (1) below), following \cite{Ozawa_1991}, we introduce the following  another asymptotic profile $\widetilde u_{\mathrm p}$ depending also on the short-range part of the nonlinearity: 
\begin{align}
\label{widetilde_u_p}
\widetilde u_{\mathrm p}=\mathcal M(t)\mathcal D(t)\widetilde w_{\mathrm p},\quad \widetilde w_{\mathrm p}(t,x)=e^{-i\lambda_1 |\varphi(x)|^2\log |t|-i\frac{\lambda_2|\varphi(x)|^{2\sigma}}{1-\sigma} t^{1-\sigma}}\varphi(x),
\end{align}
where $\widetilde w_{\mathrm p}$ satisfies $|\widetilde w_{\mathrm p}(t,x)|=|\varphi(x)|$ and 
\begin{align}
\label{widetilde_w_p}
i\partial_t \widetilde w_{\mathrm p}=\lambda_1 t^{-1}|\widetilde w_{\mathrm p}|^{2}\widetilde w_{\mathrm p}+\lambda_2 t^{\sigma-2}|\widetilde w_{\mathrm p}|^{2\sigma}\widetilde w_{\mathrm p},\quad t>0,\ x\in \R. 
\end{align}


Now we state the main result. 
\begin{theorem}
\label{theorem_1} Let  $\lambda_1>0$, $\lambda_2\in \R$ and $1<\sigma<2$. Suppose $\varphi\in H^{1+\ep}(\R)$ with some $\ep>0$, $2/3\le \alpha<1$ and $0<\beta<\min\{\ep/2,1/2\}$. 
Then, there exists a global solution $u\in C(\R;L^2(\R))$ to \eqref{NLS} satisfying $e^{itH_0}u\in C(\R;\mathcal F H^1(\R))$ and the prescribed asymptotic condition as $t\to \infty$: 
\begin{align}
\label{theorem_1_1}
\|xe^{itH_0}\{u(t)-\widetilde u_{\mathrm p}(t)\}\|_{L^2}+t^{\frac \alpha 2}\|u(t)-\widetilde u_{\mathrm p}(t)\|_{L^2}\lesssim t^{-\beta},
\end{align}
where the solution is unique in the following sense: if $u_1,u_2\in C(\R;L^2(\R))$ are two solutions to \eqref{NLS} such that $e^{itH_0}u_j\in C(\R;\mathcal F H^1(\R))$ and \eqref{theorem_1_1} hold with some $2/3\le \alpha_j<1$ and $\beta_j>0$ for $j=1,2$, respectively, then $u_1\equiv u_2$. Moreover, we have the following statements:
\begin{itemize}
\item For any $\gamma<\min\{\beta,\sigma-1\}$, the solution $u$ satisfies
\begin{align}
\label{theorem_1_2}
\|\<x\>e^{itH_0}\{u(t)- u_{\mathrm p}(t)\}\|_{L^2}\lesssim t^{-\gamma},\quad t\to \infty, 
\end{align}
where $\<x\>=\sqrt{1+|x|^2}$. 
\item The modified wave operator $W_+:\F H^{1+\ep}(\R)\ni \F^{-1}\varphi\mapsto u(0)\in \mathcal F H^1(\R)$ 
is well-defined. 
\end{itemize}
The analogous result also holds for the negative time direction $t\to -\infty$. 
\end{theorem}

\begin{remark}
\label{remark_theorem_1}$ $
\begin{itemize}
\item[(1)]  For $1<\sigma<2$, \eqref{theorem_1_2} shows the existence of a global solution $u$ to \eqref{NLS}, which scatters to $u_{\mathrm p}$ as $t\to \infty$. For $4/3<\sigma<2$, we also have a unique global solution $u$ to \eqref{NLS} scattering to $u_{\mathrm p}$ as $t\to \infty$. Indeed, by using the relations $
e^{-itH_0}xe^{itH_0}=x+it\partial_x$, $(x+it\partial_x)\mathcal M(t)\mathcal D(t)=\mathcal M(t)\mathcal D(t)i\partial_x
$, we have
\begin{align}
\label{w_p-widetilde_w_p_1}
\|u_{\mathrm p}(t)-\widetilde u_{\mathrm p}(t)\|_{L^2}=\|w_{\mathrm p}(t)-\widetilde w_{\mathrm p}(t)\|_{L^2}
&\lesssim t^{1-\sigma},\\ 
\label{w_p-widetilde_w_p_2}
\|xe^{itH_0}\{u_{\mathrm p}(t)-\widetilde u_{\mathrm p}(t)\}\|_{L^2}=\|\partial_x\{w_{\mathrm p}(t)-\widetilde w_{\mathrm p}(t)\}\|_{L^2}
&\lesssim t^{1-\sigma}\log t. 
\end{align}
Theorem \ref{theorem_1}, combined with these two estimates, shows that  there exists a global solution $u$ to \eqref{NLS} satisfying,  with some $0<\beta_0<\min\{\ep/2,1/2\}$, 
$$
\|xe^{itH_0}\{u(t)- u_{\mathrm p}(t)\}\|_{L^2}+t^{1/3}\|u(t)- u_{\mathrm p}(t)\|_{L^2}\lesssim t^{-\beta_0}+t^{1-\sigma}\log t+t^{4/3-\sigma}
$$
as $t\to \infty$. This estimate, together with \eqref{w_p-widetilde_w_p_1} and \eqref{w_p-widetilde_w_p_2}, implies \eqref{theorem_1_1} with $\alpha=2/3$ and $0<\beta<\min\{\beta_0,4/3-\sigma\}$ . Hence, we obtain the uniqueness again by Theorem \ref{theorem_1}. 

On the other hand, for $1<\sigma\le 4/3$, we do not know  the uniqueness of the solution $u$ scattering to $u_{\mathrm p}$ due to to the restriction $\alpha\ge 2/3$. 
This is the reason to introduce the asymptotic profile $\widetilde u_{\mathrm p}$. We expect that this is a technical issue since the usual wave operator is known to exist if $\lambda_1=0$ (\cite{GOV_1994}). Therefore, in principle, the effect by $\lambda_2|u|^{2\sigma}u$ should be negligible as $t\to \infty$ even if the long-range term $\lambda_1|u|^2u$ is present. 
\item[(2)] The existence of a unique global solution $u\in C([T,\infty);L^2(\R))$ satisfying $e^{itH_0}u\in  C([T,\infty);\mathcal FH^1(\R))$ and \eqref{theorem_1_1} with sufficiently large $T>0$ holds for all $\sigma>1$. 
The restriction $\sigma<2$ is due to the use of the global well-posedness in $L^2(\R)$ and persistence of the $\mathcal FH^1$-regularity of the Cauchy problem of \eqref{NLS} to extend $u$ backward in time, showing that  $u\in C(\R;L^2(\R))$ and $e^{itH_0}u\in  C(\R;\mathcal FH^1(\R))$. For $2\le \sigma$, we expect that the above theorems still hold for $\varphi\in H^{1}\cap \mathcal FH^{1+\ep}$ with \eqref{theorem_1_1} replaced by
\begin{align*}
\|\<x\>e^{itH_0}\{u(t)-u_{\mathrm p}(t)\}\|_{L^2}+\|\partial_x\{u(t)-u_{\mathrm p}(t)\}\|_{L^2}
 +t^{\frac \alpha 2}\|u(t)-u_{\mathrm p}(t)\|_{L^2}.
\end{align*}
However we do not pursue this issue further here for the sake of simplicity. 
\item[(3)] The solution $u$ satisfies the same $L^\infty$-decay estimate as for the free solution as $t\to \infty$: \begin{align*}
\|u(t)\|_{L^\infty}\lesssim t^{-1/2},\quad t\to \infty.\end{align*}
Indeed, since $\|e^{-itH_0}\<x\>^{-1}\|_{L^2\to L^\infty}\lesssim t^{-1/2}$ and $\|\widetilde u_{\mathrm p}(t)\|_{L^\infty}\le t^{-1/2}\|\varphi\|_{L^\infty}$, we have$$\|u(t)\|_{L^\infty}\lesssim \|e^{-itH_0}\<x\>^{-1}\<x\> e^{itH_0}\{u(t)-\widetilde u_{\mathrm p}(t)\}\|_{L^\infty}+\|\widetilde u_{\mathrm p}(t)\|_{L^\infty}\lesssim t^{-1/2-\beta}+t^{-1/2}.$$
\item[(4)] The main ingredient in the proof of Theorem \ref{theorem_1} is to introduce a new formulation of the FSP for \eqref{NLS} based on the linearization of \eqref{NLS} around a given  asymptotic profile. For the linearized equation, which is a system of Schr\"odinger equations with non-symmetric, time-dependent and long-range potentials, we prove a modified energy identity and an associated energy estimate for the linearized equation, which allow us to apply a rather simple energy method to construct the modified wave operator.  In particular, our argument does not rely on either the complete integrability or the smoothness of the nonlinearity  of the cubic NLS. 
\item[(5)] Combining Theorem \ref{theorem_1} with the result by \cite{Deift_Zhou_2003} on the large data modified scattering for the CP of \eqref{NLS}, one can also construct the modified scattering operator for arbitrarily large scattering data if $\lambda_2=0$. Precisely, it has been shown by \cite[Theorems 4.9 and  4.10]{Deift_Zhou_2003} that for any $u_0\in \mathcal F H^1$, there exists a unique solution $u\in C(\R; L^2\cap L^\infty)$ to \eqref{NLS} with the initial condition $u(0)=u_0$ and a unique $u_+\in L^2\cap L^\infty$ such that $\|u(t)-u_{\mathrm p}(t)\|_{L^\infty}\to 0$ as $t\to \infty$\footnote{To be more precise, \cite{Deift_Zhou_2003} has established this statement for the equation $i\partial_t u+\partial_x^2u=2|u|^2u$. However, it can be easily translated to  \eqref{NLS} with $\lambda_2=0$ by the scaling $u\mapsto u(\lambda_1t/2,\sqrt{\lambda_1}x)$ with $\lambda_1>0$. }. This, together with the existence of the negative time modified wave operator $W_-$ provided by Theorem \ref{theorem_1}, shows that the modified scattering operator $S:\mathcal FH^{1+\ep}\ni u_-\mapsto u_+\in L^2\cap L^\infty$ is well-defined. 
\end{itemize}
\end{remark}

The defocusing condition $\lambda_1>0$ is essential in our argument, so we do not know whether Theorem \ref{theorem_1} holds for the focusing case. However, we still obtain an explicit upper bound for the size of scattering data $\varphi$ to ensure the modified scattering in the focusing case, which improves upon a part of an earlier result by \cite{Carles_2001} (see Remark \ref{remark} (2) for more details): 


\begin{theorem}
\label{theorem_2}
Suppose $\lambda_1,\lambda_2\in \R$, $1<\sigma<2$, $\varphi\in H^{1+\ep}(\R)$ with some $\ep>0$ and 
$$
|\lambda_1|\|\varphi\|_{L^\infty(\R)}^2<1. 
$$
Let $\max\{1,2|\lambda_1|\|\varphi\|_{L^\infty(\R)}^2\}<\alpha<2$ and $0<\beta<\min\{\ep/2,1-\alpha/2\}$. Then, there exists a unique global solution $u\in C(\R;L^2(\R))$ to \eqref{NLS} satisfying $e^{itH_0}u\in C(\R;\mathcal F H^1(\R))$ and the prescribed asymptotic condition \eqref{theorem_1_1} as $t\to \infty$. In particular, the modified wave operator
$$
W_+:\{u_+\in \mathcal FH^{1+\ep}(\R)\ |\ |\lambda_1|\|\widehat{u_+}\|_{L^\infty}^2<1\}\ni \F^{-1}\varphi\mapsto u(0)\in \mathcal FH^1(\R)
$$
is well-defined. The analogous result also holds for the negative time $t\to -\infty$. 
\end{theorem}

\begin{remark}
\label{remark}
$ $
\begin{itemize}
\item[(1)] The same remarks as Remark \ref{remark_theorem_1} (1), (2) and (3) also apply to Theorem \ref{theorem_2}. For instance, if $0<\gamma<\min\{\beta,\sigma-1\}$, then Theorem \ref{theorem_2} ensures the existence of a global solution $u$ satisfying
$$
\|xe^{itH_0}\{u(t)-u_{\mathrm p}(t)\}\|_{L^2}\lesssim t^{-\gamma},\quad t\to \infty.
$$
If in addition $4/3<\sigma<2$, we also have the uniqueness of such a solution $u$ scattering to $u_{\mathrm p}$ as $t\to \infty$. 

\item[(2)] The author of  \cite[Corollary 1]{Carles_2001} constructed the modified wave operator $$W_+:\{u_+\in H^3\cap \<x\>^{-1}H^2\ |\ |\lambda_1|\|\widehat {u_+}\|_{L^\infty(\R)}^2<1/2\} \ni \F^{-1}\varphi\mapsto u(0)\in H^1\cap \mathcal FH^1.$$ While the topology of the scattering is stronger than that in Theorem \ref{theorem_2}, our result improves the regularity and smallness conditions on $\varphi$. 
\end{itemize}
\end{remark}

\subsection{Idea of the proof}
Here we describe the idea of the proof of Theorem \ref{theorem_1} with explaining the difficulty of the large data problem. 
In what follows, we denote $\|f\|=\|f\|_{L^2(\R)}$.  Define for short $\sigma_1=1$ and $\sigma_2=\sigma$ so that 
\begin{align}
\lambda_1|u|^2u+\lambda_2 |u|^{2\sigma}  u=\sum_{j=1}^2\lambda_j|u|^{2\sigma_j}u. 
\end{align}
Instead of $u,\widetilde u_{\mathrm p}$, it is convenient to work with their pseudo-conformal transforms defined by
\begin{align}
\label{pseudo_1}
v(t,x)&:=\overline{\mathcal T[\mathcal M(t)\mathcal D(t)]^{-1}u(t,x)},\\ 
\label{pseudo_2}
v_{\mathrm p}(t,x)&:=\overline{\mathcal T[\mathcal M(t)\mathcal D(t)]^{-1}\widetilde u_{\mathrm p}(t,x)}=e^{-i\lambda_1 |\varphi(x)|^2\log |t|-i\frac{\lambda_2|\varphi(x)|^{2\sigma}}{\sigma-1} t^{\sigma-1}}\overline{\varphi(x)},
\end{align}
where $\mathcal Tf(t):=f(1/t)$. By \eqref{NLS}, \eqref{widetilde_w_p}, \eqref{Dollard_1} and a direct computation, we see that they satisfy
\begin{align}
\label{v}
(i\partial_t-H_0)v&=\sum_{j=1}^2\lambda_j t^{\sigma_j-2}|v|^{2\sigma_j}v,\\
\label{v_p}
i\partial_t v_{\mathrm p}&=\sum_{j=1}^2\lambda_j t^{\sigma_j-2}|v_{\mathrm p}|^{2\sigma_j}v_{\mathrm p}. 
\end{align}
Moreover, we have
$$
\|u(t)\|=\|\{\mathcal M(t)\mathcal D(t)\}^{-1}u(t)\|,\quad \|xe^{itH_0}u(t)\|=\|\partial_x\{\mathcal M(t)\mathcal D(t)\}^{-1}u(t)\|
$$
and hence the asymptotic condition \eqref{theorem_1_1} is equivalent to 
\begin{align}
\label{v-v_p}
\|\partial_xv(t)-\partial_xv_{\mathrm p}(t)\|+t^{-\alpha/2}\|v(t)-v_\mathrm p(t)\|\lesssim t^\beta,\quad t\to +0. 
\end{align}
It follows from \eqref{v} and \eqref{v_p} that $v-v_{\mathrm p}$ solves
\begin{align}
\label{equation_v_1}
\left(i\partial_t -H_0\right)(v-v_{\mathrm p}-\mathcal R(t) v_{\mathrm p})
=\sum_{j=1}^2\lambda_j t^{\sigma_j-2}\left\{ |v|^{2\sigma_j}v-|v_{\mathrm p}|^{2\sigma_j}v_{\mathrm p}-\mathcal R(t)|v_{\mathrm p}|^{2\sigma_j}v_{\mathrm p}\right\},
\end{align}
where $\mathcal R(t)f=e^{-itH_0}f-f$ satisfies $\|\mathcal R(t)f\|\lesssim t^{\delta}\|f\|_{H^{2\delta}}$ for $0\le \delta\le1$ (see Lemma \ref{lemma_R}). Indeed, 
\begin{align*}
(i\partial_t-H_0)v_{\mathrm p}
&=e^{-itH_0}i\partial_te^{itH_0}v_{\mathrm p}=e^{-itH_0}\left\{i\partial_t v_{\mathrm p}+i\partial_t (e^{itH_0}-I)v_{\mathrm p}\right\}\\
&=\left\{I+e^{-itH_0}-I\right\}\sum_{j=1}^2\lambda_j t^{\sigma_j-2}|v_{\mathrm p}|^{2\sigma_j}v_{\mathrm p}+e^{-itH_0} i\partial_te^{itH_0}(I-e^{-itH_0})v_{\mathrm p}\\
&=\sum_{j=1}^2\lambda_j t^{\sigma_j-2}\left\{|v_{\mathrm p}|^{2\sigma_j}v_{\mathrm p}+\mathcal R(t)|v_{\mathrm p}|^{2\sigma_j}v_{\mathrm p}\right\}-(i\partial_t -H_0)\mathcal R(t)v_{\mathrm p}
\end{align*}
and hence \eqref{equation_v_1} holds. 
We set  for short $v_*=v-v_{\mathrm p}$. By virtue of \eqref{v-v_p} and \eqref{equation_v_1}, it is natural to consider the following integral equation: 
\begin{align}
v_*(t)=\mathcal R(t)v_{\mathrm p}(t)-i \sum_{j=1}^2\lambda_j \int_0^t e^{-i(t-s)H_0}s^{\sigma_j-2}\left\{ |v|^{2\sigma_j}v-|v_{\mathrm p}|^{2\sigma_j}v_{\mathrm p}-\mathcal R(s)|v_{\mathrm p}|^{2\sigma_j}v_{\mathrm p}\right\}ds,
\label{IE}
\end{align}
where the difference $|v|^{2\sigma_j}v-|v_{\mathrm p}|^{2\sigma_j}v_{\mathrm p}$ satisfies
\begin{align}
\left||v|^{2\sigma_j}v-|v_{\mathrm p}|^{2\sigma_j}v_{\mathrm p}\right|
\label{F(v)-F(v_p)}
\lesssim |\varphi|^{2\sigma_j}|v_*|+|v_*|^{2\sigma_j+1}.
\end{align}
Then one can solve \eqref{IE}, for instance,  in the energy space
$$
\{v_*\in C((0,T];H^1(\R))\ |\ \sup_{0<t\le T}t^{-\beta}\left(\|\partial_x v_*(t)\|+t^{-\alpha} \|v_*(t)\|\right)<\infty\}
$$
for sufficiently small $T$ and some $0<\alpha,\beta<1/2$ as long as $|\lambda_1|\|\varphi\|_{L^\infty}^2$ is sufficiently small. This type formulation (with or without the use of pseudo-conformal transform) has been employed in the literature on the small data modified scattering (see for instance \cite{Ozawa_1991,Hayashi_Naumkin_2006,Lindblad_Soffer_2006}). However, this argument does not work for the large data problem. An obstruction is the first term $|\varphi|^{2}|v_*|$ in the RHS of \eqref{F(v)-F(v_p)} with $\sigma_1=1$ since the integral 
$
\int_0^t s^{-1}\|v_*(s)\|_{H^1} ds
$ 
decays as $t\to +0$ with at most the same rate as that of $\|v_*\|_{H^1}$ and thus cannot be absorbed in the LHS of \eqref{IE}. 

To overcome this difficulty, we introduce a new formulation in which the first order term is regarded as a linear potential term. For a technical reason, we also regard the first order term of the short-range part as a linear potential term. Precisely, we extract the first order term from the difference $\sum_{j=1}^2\lambda_j t^{\sigma_j-2}\left(|v|^{2\sigma_j}v-|v_{\mathrm p}|^{2\sigma_j}v_{\mathrm p}\right)$ by Taylor's expansion, which leads the following system of nonlinear Schr\"odinger equations: 
\begin{align}
\label{Equation}
(i\partial_t-\mathcal H(t))(\vec v_*-\vec e_1)=\sum_{j=1}^2\lambda_j t^{\sigma_j-2}\mathcal J\vec G_{\sigma_j}[v_*]+\mathcal J\vec e_{2},
\end{align}
where $\vec v_*=(v_*,\overline v_*)^{\mathrm T}$ is $\C^2$-valued, $\mathcal J=\diag(1,-1)$, $\vec e_j=(e_j,\overline{e_j})^{\mathrm T}$ for $j=1,2$ are error terms, $\sum_{j=1}^2\lambda_j t^{\sigma_j-2}\mathcal J\vec G_{\sigma_j}[v_*]$ is a new nonlinear term which consists of $\sum_{j=1}^2\lambda_j t^{\sigma_j-2}\left\{ |v|^{2\sigma_j}v-|v_{\mathrm p}|^{2\sigma_j}v_{\mathrm p}\right\}$ minus its first order term.  The Hamiltonian $\mathcal H(t)$ is of the form $\mathcal H_0+\mathcal V(t,x)$, where $\mathcal H_0=\diag(H_0,-H_0)$ is the matrix-valued free Schr\"odinger operator and the potential $\mathcal V(t,x)$, which consists of coefficients of the first order part of $\sum_{j=1}^2\lambda_j t^{\sigma_j-2}\left\{ |v|^{2\sigma_j}v-|v_{\mathrm p}|^{2\sigma_j}v_{\mathrm p}\right\}$, is a non-symmetric and time-dependent potential of long-range type satisfying $\mathcal V(t,x)=O(t^{-1})$ as $t\to +0$. 
The main advantage of this equation compared with \eqref{equation_v_1} is that we have
$$
\sum_{j=1}^2\left|\lambda_j t^{\sigma_j-2}\mathcal J\vec G_{\sigma_j}[v_*]\right|\lesssim t^{-1}(|v_*|^2+|v_*|^3)+t^{\sigma-2}(|v_*|^2+|v_*|^{2\sigma+1})
$$ from which one can expect that the new nonlinear term decays faster than $v_*$ as $t\to +0$. Hence, if the propagator $\mathcal U(t,s)$ generated by $\mathcal H(t)$, {\it i.e.}, the solution operator for the linearized equation 
\begin{align}
\label{idea_0}
(i\partial_t-\mathcal H(t))\bvec \Psi=0,
\end{align} satisfies a good energy estimate in a suitable Sobolev space algebra, then \eqref{Equation} can be solved by a standard energy method. To this end, we introduce a modified energy
\begin{align}
\label{Q_1}
Q_\alpha[\psi](t)&=\frac{\|\partial_x  \psi(t)\|^2}{4}+t^{-\alpha} \|\psi(t)\|^2+\lambda_1  t^{-1} \||\varphi|^{\sigma-1}\Re[\overline{v_{\mathrm p}(t)}\psi(t)]\|^2
\end{align}
and show the following energy estimate
\begin{align}
\label{Idea_1}
Q_\alpha[\mathcal U(t,s)\vec \psi_0](t)\lesssim Q_\alpha[ \psi_0](s),\quad 0<s\le t\le1,
\end{align}
with the initial data $\vec\psi_0=(\psi_0,\overline{\psi_0})^{\mathrm T}$, where $Q_\alpha[(f_1, f_2)^{\mathrm T}]:=Q_\alpha[f_1]+Q_\alpha[\overline{f_2}]$. This is the crucial step in the proof and maybe the most important contribution of the present paper. It is worth mentioning that $\mathcal U(t,s)$ does not preserve the $L^2$-norm since $\mathcal H(t)$ is not symmetric. Thus, even dealing with the $L^2$-bound  for $\mathcal U(t,s)$ is highly non-trivial. The proof of \eqref{Idea_1} essentially relies on the following modified energy identity:
\begin{align}
\frac{d}{dt}\widetilde{Q}_{\alpha}[\psi]
\nonumber
&=
-\alpha t^{-\alpha-1} \|\psi\|^2
-\sum_{j=1}^2\lambda_j\sigma_j(2-\sigma_j) t^{\sigma_j-3}\||\varphi|^{\sigma_j-1}\Re[\overline{v_{\mathrm p}}\psi]\|^2\\
\nonumber
&\quad -4\sum_{j=1}^2\lambda_j\sigma_j t^{\sigma_j-2-\alpha}\<|\varphi|^{2\sigma_j-2}\Re[\overline{v_{\mathrm p}}\psi],\Im[\overline{v_{\mathrm p}}\psi]\>\\
\label{identity}
&\quad 
-\sum_{j=1}^2\lambda_j\sigma_j  t^{\sigma_j-2}\Im\<\partial_x  \psi,|\varphi|^{2\sigma_j-2}\Re[\varphi\overline{\partial_x  \varphi}]\psi\>
\end{align}
for any $H^1$-solution $\bvec\Psi(t)=(\psi(t),\overline{\psi(t)})^{\mathrm T}$ to \eqref{idea_0}, where $$\widetilde{Q}_{\alpha}[\psi]:=Q_\alpha[\psi]+\lambda_2\sigma_2 t^{\sigma_2-2}\||\varphi|^{\sigma_2-1}\Re[\overline{v_{\mathrm p}(t)}\psi(t)]\|^2.$$
This energy identity implies $\frac{d}{dt}\widetilde{Q}_{\alpha}[\psi]\lesssim t^{\alpha/3-1}\widetilde{Q}_{\alpha}[\psi]$ and hence \eqref{Idea_1} since $Q_\alpha[\psi]\sim \widetilde{Q}_{\alpha}[\psi]$ for sufficiently small $t>0$. 

Estimate \eqref{Idea_1}, together with suitable energy bounds for $Q_\alpha[G_{\sigma_j}[v_*]]$ and $Q_\alpha[e_j]$, allows us to apply a simple energy method to show Theorem \ref{theorem_1}. Specifically, we consider the energy space 
$$
\{\vec v_*=(v_*,\overline{v_*})^\mathrm T\in C((0,T];H^1(\R)\times H^1(\R))\ |\ \sup_{0<t\le T}t^{-\beta}\left(\|\partial_xv_*(t)\|+t^{-\alpha/2}\|v_*(t)\|\right)<\infty\}
$$
and show that \eqref{Equation} subjected to the condition $\|\partial_xv_*(t)\|+t^{-\alpha/2}\|v_*(t)\|\lesssim t^{\beta}$ as $t\to +0$ admits a unique global solution under the assumptions on the parameter $\alpha$ and $\beta$ stated in Theorem \ref{theorem_1}. Once a unique global solution to \eqref{Equation} is obtained, its inverse pseudo-conformal transform gives a unique global solution to the original NLS \eqref{NLS} satisfying the asymptotic condition \eqref{theorem_1_1}. 

\begin{remark}[Some open problems]
\label{remark_future}
We expect that the method in this paper has a potential application to some other NLS with long-range nonlinearities in one space dimension. For example, if we consider the NLS with subcritical nonlinearities of the form
$$
i\partial_t u-H_0u=|u|^{2\sigma}u,\quad x\in \R,\quad t\in \R,\quad 0<\sigma<1,
$$
then, by an essentially same argument, we can prove a similar energy estimate as \eqref{Idea_1} for the solution to the associated linearized equation provided $2/3<\sigma<1$ and $\varphi$ is sufficiently smooth. However, this is not enough and there are several difficulties in handling nonlinear terms (see, for instance, Remark \ref{remark_G} below). This is left for future work. 

It would be also interesting whether it applies to the cubic NLS with a linear long-range potential (see \cite{Kawamoto_Mizutani_PAMS} for the small-data setting), or the system of cubic NLS under appropriate structural conditions on the nonlinearity such as the Manakov system \cite{Manakov}. 

To extend the present method to two or three space dimensions, we need to find a higher-order counterpart of the modified energy \eqref{Q_1}, which is unclear. 
\end{remark}

\subsection{Comparison with the Hartree equation}
\label{section_Hartree}
As mentioned in the introduction, the modified scattering for the Hartree equations \eqref{Hartree} has been established by 
\cite{Ginibre_Velo_2000_1,Ginibre_Velo_2000_2,Ginibre_Velo_2001,Ginibre_Velo_2014,Ginibre_Velo_2015,Nakanishi_CPAA,Nakanishi_AHP} without size restriction on scattering data. Here we compare our method with those of \cite{Nakanishi_CPAA,Nakanishi_AHP, Ginibre_Velo_2014,Ginibre_Velo_2015} in the critical case $\sigma=1$.  It should be stress that these methods also work well for the subcritical case $0<\sigma<1$, regardless the defocusing or focusing cases. 

Let $g(u)=\lambda |x|^{-1}*|u|^2$ for short. By the pseudo-conformal transform \eqref{pseudo_1}, \eqref{Hartree} with $\sigma=1$ is transformed into
\begin{align}
\label{Hartree_1}
i\partial_t v+\frac12\Delta v=t^{-1}g(v)v.
\end{align}

On one hand, the author of \cite{Nakanishi_CPAA,Nakanishi_AHP}  considered the function $$ v_1=e^{ig(\varphi)\log t}e^{-it\Delta/2}v.$$ Then \eqref{Hartree_1} is equivalent to
$$
i\partial_t v_1+V_1(v) v_1=0,
$$
where 
$
V_1(v)=t^{-1}e^{-ig(\varphi)\log t}\left\{e^{-it\Delta/2}g(v)e^{it\Delta/2}-g(\varphi)\right\}e^{ig(\varphi)\log t}
$. Then the unique solution $v_1$ satisfying $v_1\to \varphi$ as $t\to 0$ is constructed as the limit of the sequence $(w_k)$ defined iteratively by solving the following system: 
$$
i\partial_t  w_k+V_1(w_{k-1}) w_k=0,\ 0<t\ll1;\quad w_k(0)=\varphi.
$$
This argument is quite similar to solving the linearized equation $$i\partial_t  v_1+V_1(f) v_1=0,\ 0<t\ll1;\quad  v_1(0)=\varphi.$$ and showing that the composition of maps $v_1\mapsto e^{it\Delta/2}e^{-ig(\varphi)\log t}f$ and $f\mapsto  v_1$ is a contraction, where  $f$ is a given function belonging to the same energy space as that for the solution $ v$. 

On the other hand, the authors of \cite{Ginibre_Velo_2014,Ginibre_Velo_2015} dealt with the function $$ v_2=e^{i g_S(\varphi)\log t}v,$$ where $ g_S(\varphi)=\chi(t^{1/2}|D|)g(\varphi)$ and $\chi\in C^\infty([0,2))$ with $\chi\equiv1$ on $[0,1]$. Roughly speaking $g_S$ is the restriction of $g$ near the frequency region $t^{1/2}|\xi|\lesssim1$. Then \eqref{Hartree_1} becomes $$i\partial_t  v_2+L(v) v_2=0,$$ where
$
L( v)=\frac12(\nabla-iA)^2+V_2(v)$,  $V_2(v)=-t^{-1}(g( v)- g_S(\varphi))+t^{-1/2}(\log t)|D|\chi'(t^{1/2}|D|)g(\varphi)
$ and $A=\nabla \{e^{-i g_S(\varphi)\log t}\varphi\}$. 
As for \cite{Nakanishi_CPAA,Nakanishi_AHP}, an important step is to solve the linearized equation: 
$$
i\partial_t  v_2=L(f)v_2,\ 0<t\ll1;\quad  v_2(0)=\varphi.
$$

Therefore, it turns out that suitable energy estimates for the above linearized equations play crucial roles in both methods. For that purpose, both methods employ the  smoothing property (in the high frequency region) of the convolution such as
$$
\|
|x|^{-1}*h\|_{\dot H^s}=C_{d}\||D|^{1-d}h\|_{\dot H^s}\lesssim \|h\|_{\dot H^{s+1-d}},\quad h\in H^{s+1-d}(\R^d),\quad d>1,
$$ 
for showing the integrability in $t\in (0,1]$ of the potential terms $V_1(f)$ or $V_2(f)$. Indeed,  combining with the estimate $\|(e^{it\Delta/2}-1)h\|\lesssim t^\delta \|h\|_{\dot H^{2\delta}}$ in the case of the former method, or $\|(1-\chi(t^{1/2}|D|)|h\|\lesssim t^\delta \|h\|_{\dot H^{2\delta}}$ for the latter one, this smoothing property can be used to obtain an additional time-decay as $t\to 0$ without loss of derivatives (see \cite[Lemma 5.1]{Nakanishi_CPAA} and \cite[Proposition 3.2]{Ginibre_Velo_2014} and their proofs).

However, such a smoothing property (and hence an additional time-decay without loss of derivatives) is unavailable for the power-type nonlinearities. Therefore, applying these methods directly to the present setting seems difficult. Instead, we renormalize the non-integrable part of the nonlinearity as a time-dependent potential term $\mathcal V(t,x)\vec v_*$ and incorporate it into the linear part. Under the defocusing condition $\lambda_1>0$, it becomes a positive definite (and thus good) term in the modified energy for the linearized equation (see the last term in \eqref{Q_1}), even though it is still not integrable in $t$. As far as we know, this viewpoint is new to the study of modified scattering for the NLS. 

\subsection{Organization of the paper}

The rest of the paper is devoted to the proof of Theorems \ref{theorem_1} and \ref{theorem_2}. We first explain the linearization of \eqref{NLS} and derivation of the integral equation in Section \ref{subsection_integral_equation}. The energy estimate for the modified energy $Q_{\alpha}$ is proved in Section \ref{subsection_energy}. The proofs of Theorems \ref{theorem_1} and \ref{theorem_2} are given in Sections \ref{section_proof_Theorem_1} and \ref{section_proof_Theorem_2}, respectively.

\subsection{Notation}
\label{notation}
Here we summarize notations used in this paper: 
\begin{itemize}
\item $\<x\>=\sqrt{1+|x|^2}$. 
\item $L^p(\R)$ and $H^s(\R)$ denote the Lebesgue and $L^2$-Sobolev spaces, respectively. 
\item $\<f,g\>=\int_\R f\overline gdx$ denotes the inner product in $L^2(\R)$. 
\item $\|f\|=\|f\|_{L^2(\R)}$ for $f\in L^2(\R)$. 
\item $\mathcal L^2(\R)=L^2(\R)\times L^2(\R)$ and $\mathcal H^s(\R)=H^s(\R)\times H^s(\R)$. 
\item $\mathbb B(X)$ denotes the space of bounded operators on $X$. 
\item For $a\in \C$, $\vec a$ denotes the vector $(a,\overline a)^{\mathrm T}\in \C^2$, the transposed vector of $(a,\overline a)$. 
\item $A\lesssim B$ (resp. $A\gtrsim B$) means there exists a non-essential constant $C>0$ such that $A\le CB$ (resp. $A\ge CB$). 
\end{itemize}

\section{Preliminaries}

\subsection{Integral equation}
\label{subsection_integral_equation}
As in the standard argument, we first rewrite the NLS \eqref{NLS} subject to the condition $\|u(t)-\widetilde u_{\mathrm p}(t)\|\to 0$  as $t\to \infty$ as an appropriate integral equation. To this end, we assume for a while $u$ is a smooth solution to \eqref{NLS}. Let $v$ and $v_{\mathrm p}$ be the pseudo-conformal transforms of $u$ and $\widetilde u_{\mathrm p}$ defined by \eqref{pseudo_1} and \eqref{pseudo_2}, respectively, satisfying \eqref{equation_v_1}. To extract the first order term with respect to $v_*$ from the nonlinear term, we use the following: 
\begin{lemma}
\label{lemma_Taylor}
For all $z_0,z_1\in \C$, 
$$
|z_1|^{2\sigma}z_1=|z_0|^{2\sigma}z_0+(\sigma+1)|z_0|^{2\sigma}z_{*}+\sigma|z_0|^{2\sigma-2}z_0^2\overline{z_{*}}+G_\sigma[z_{*}],
$$
where $z_{*}:=z_1-z_0$ and $G_{\sigma}[z_{*}]$ is defined by, with $z_\theta=z_0+\theta z_{*}$, 
\begin{align*}
G_{\sigma}[z_{*}]=(\sigma+1)z_{*}\int_0^1\left(|z_\theta|^{2\sigma}-|z_0|^{2\sigma}\right)d\theta+\sigma\overline{z_{*}}\int_0^1\left(|z_\theta|^{2\sigma-2}z_\theta^2-|z_0|^{2\sigma-2}z_0^2\right)d\theta. 
\end{align*}
\end{lemma}

\begin{proof}
The result follows from Taylor's formula
$$
f(z_1)=f(z_0)+z_*\int_0^1f_z(z_\theta)d\theta+\overline{z_*}\int_0^1f_{\overline z}(z_\theta)d\theta,
$$
where $f_z=(\partial_xf-i\partial_yf)/2$ and $f_{\overline z}=(\partial_xf+i\partial_yf)/2$ for $z=x+iy$. 
\end{proof}


Recall that $v_*=v-v_{\mathrm p}$. It follows from this lemma that \eqref{equation_v_1} is rewritten as
\begin{align}
\nonumber
&\left(i\partial_t -H_0\right)(v_*-\mathcal R(t) v_{\mathrm p})\\
\label{equation_v_2}
&=\sum_{j=1}^2\lambda_j t^{\sigma_j-2}\left\{(\sigma_j+1)|\varphi|^{2\sigma_j}v_*+\sigma_j |\varphi|^{2\sigma_j-2}v_{\mathrm p}^2\overline{v_*}+G_{\sigma_j}[v_*]-\mathcal R(t)|v_{\mathrm p}|^{2\sigma_j}v_{\mathrm p}\right\},
\end{align}
where, with $v_\theta=v_{\mathrm p}+\theta v_*$, 
\begin{align}
\label{G_sigma_j}
G_{\sigma_j}[v_*]=(\sigma_j+1)v_{*}\int_0^1\left(|v_\theta|^{2\sigma_j}-|v_{\mathrm p}|^{2\sigma_j}\right)d\theta+\sigma_j\overline{v_{*}}\int_0^1\left(|v_\theta|^{2\sigma_j-2}v_\theta^2-|v_{\mathrm p}|^{2\sigma_j-2}v_{\mathrm p}^2\right)d\theta. 
\end{align}
Now we consider the following linearized equation
\begin{align}
\label{equation_psi_1}
i\partial_t \psi -H_0 \psi =\sum_{j=1}^2\lambda_j t^{\sigma_j-2}\left\{(\sigma_j+1)|\varphi|^{2\sigma_j}\psi+\sigma_j |\varphi|^{2\sigma_j-2}v_{\mathrm p}^2\overline{\psi}\right\},
\end{align}
where we will use in the next subsection that the RHS can be written as
\begin{align}
\label{equation_psi_2}
\sum_{j=1}^2\lambda_j t^{\sigma_j-2}\left(|\varphi|^{2\sigma_j}\psi+2\sigma_j |\varphi|^{2\sigma_j-2}\Re[\overline{v_\mathrm p}\psi]v_{\mathrm p}\right). 
\end{align}
Since \eqref{equation_psi_1} is not $\C$-linear, we make it as a linear system in a usual way as follows. Define
\begin{align}
\mathcal H(t)=\mathcal H_0+\mathcal V=\begin{pmatrix}H_0&0\\0&-H_0\end{pmatrix}+\sum_{j=1}^2\lambda_j t^{\sigma_j-2}\begin{pmatrix}(\sigma_j+1)|\varphi|^{2\sigma_j}& \sigma_j |\varphi|^{2\sigma_j-2}v_{\mathrm p}(t)^2\\-\sigma_j |\varphi|^{2\sigma_j-2}\overline{v_{\mathrm p}(t)}^2&-(\sigma_j+1)|\varphi|^{2\sigma_j}\end{pmatrix}.
\end{align}
Then we arrive at the linearized system of \eqref{equation_v_2} around the profile $v_\mathrm p$: 
\begin{align}
\label{linearized_equation}
i\partial_t \bvec\Psi -\mathcal H(t)\bvec\Psi=0,
\end{align}
where $\bvec \Psi=\bvec \Psi(t,x)$ is $\C^2$-valued. Note that \eqref{equation_psi_1} is equivalent to \eqref{linearized_equation} with $\bvec \Psi=(\psi,\overline \psi)^{\mathrm T}$, where $\bvec a^{\mathrm T}$ is the transposed vector of $\bvec a\in \C^2$.  Let $\mathcal U(t,s)$ be the associated propagator, namely the solution to \eqref{linearized_equation} with the initial state $\bvec\Psi(s,x)=\bvec\Psi_0(x)$ at time $s$ is given by $\bvec\Psi(t,x)=[\mathcal U(t,s)\bvec\Psi_0](x)$. Basic properties of $\mathcal U(t,s)$ used in the following argument are summarized as follows. In what follows, for a complex number $a\in \C$, we denote 
$$
\vec a=(a,\overline a)^{\mathrm T}\in \C^2,\quad \mathcal J\vec a=\begin{pmatrix}1&0\\0&-1\end{pmatrix}\vec a=(a,-\overline a)^{\mathrm T}\in \C^2, 
$$
We also set $\L^p=L^p\times L^p$ and $\mathcal H^k=H^k\times H^k$. Let  $\mathbb B(X)$ denote the space of bounded operators on $X$.

 \begin{lemma}
 \label{lemma_propagator}
 Suppose $\varphi\in H^1(\R)$. Then there exists a unique propagator 
 $\{\mathcal U(t,s)\}_{t,s\in (0,\infty)}\subset \mathbb B(\mathcal H^1(\R))$ generated by $\mathcal H(t)$ satisfying the following properties:
 \begin{itemize}
 \item[(1)] $\mathcal U(t,s)=\mathcal U(t,r)\mathcal U(r,s)$ and $\mathcal U(t,t)=I$ for all $r,s,t >0$. 
 \item[(2)] The map $(0,\infty)^2\ni (t,s)\mapsto \mathcal U(t,s)\in \mathbb B(\mathcal H^1(\R))$ is strongly continuous;
 \item[(3)] For $\bvec \Psi_0\in \mathcal H^1(\R)$ , $\bvec\Psi(t,x)=\mathcal U(t,s)\bvec\Psi_0(x)$ solves \eqref{linearized_equation} in $\mathcal H^{-1}(\R)$. 
 \item[(4)] For $\psi_0\in L^2(\R)$, $\mathcal U(t,s)\vec \psi_0$ is given by $\vec \psi(t)$, where $\psi(t)$ solves \eqref{equation_psi_1} with $\psi(s)=\psi_0$. 
\end{itemize}
\end{lemma}

\begin{proof}
We may assume $0<s<t<\infty$ without loss of generality. Taking into account that $\mathcal H_0$ is self-adjoint on $\mathcal L^2$, generating the unitary group $e^{-it\mathcal H_0}$, we observe that the solution $\bvec \Psi(t)=\mathcal U(t,s)\bvec \Psi_0$ to \eqref{linearized_equation} is given by the Duhamel formula 
\begin{align}
\label{lemma_propagator_proof_1}
\bvec\Psi(t,x)=e^{-i(t-s)\mathcal H_0}\bvec \Psi_0(x)-i\int_s^t e^{-i(t-r)\mathcal H_0}\mathcal V(r,x)\bvec \Psi(r,x)dr. 
\end{align}
To be precise, we define a sequence $\bvec \Psi_n\in C((0,\infty);\mathcal H^1(\R))$ by $\bvec \Psi_0(t):=\bvec\Psi_0$ and 
$$
\bvec \Psi_n(t,x):=e^{-i(t-s)\mathcal H_0}\bvec \Psi_0(x)-i\int_s^t e^{-i(t-r)\mathcal H_0}\mathcal V(r,x)\bvec \Psi_{n-1}(r,x)dr,\quad n\ge1. 
$$
Since $e^{-it\mathcal H_0}$ leaves Sobolev norms $\|\cdot\|_{\mathcal H^k}$ invariant for all $k\in \R$ and, for any $\ep>0$, 
\begin{align*}
\|\mathcal V(t)\|_{(L^\infty)^4}&\lesssim t^{-1}\|\varphi\|_{L^\infty}^{2}+t^{\sigma-2}\|\varphi\|_{L^\infty}^{2\sigma}\lesssim t^{-1},\\
\|\partial_x  \mathcal V(t)\|_{(L^2)^4}&\lesssim t^{-1-\ep}(1+\|\varphi\|_{L^\infty}^2)\|\partial_x  \varphi\|+t^{\sigma-2}(1+\|\varphi\|_{L^\infty}^{2\sigma})\|\partial_x  \varphi\|\lesssim t^{-1-\ep},
\end{align*}
where recalling the definition \eqref{pseudo_2} we have used the bound
$$
|\partial_x  v_{\mathrm p}|\lesssim \left(1+t^{-\ep}|\varphi|^2\right)|\partial_x  \varphi|+t^{\sigma-1}\left(1+|\varphi|^{2\sigma}\right)|\partial_x  \varphi|,
$$
we observe there exists $C>0$ independent of $s,t$ and $\bvec \Psi_0$ such that
\begin{align*}
\|\bvec \Psi_1(t)-\bvec\Psi_0\|_{\mathcal H^1}
&\le 
\int_s^t \left(\|\mathcal V(r)\|_{(L^\infty)^4}\|\bvec \Psi_{0}\|_{\mathcal H^1}+\|\partial_x  \mathcal V(r)\|_{(L^2)^4}\|\bvec \Psi_{0}\|_{\mathcal L^\infty}\right)dr\\
&\le C\|\bvec\Psi_0\|_{\mathcal H^1}\int_s^t r^{-1-\ep}dr\\
&\le Cs^{-1-\ep}(t-s)\|\bvec\Psi_0\|_{\mathcal H^1}.
\end{align*}
It follows from this estimate and the same argument that
$$
\|\bvec \Psi_2(t)-\bvec\Psi_1(t)\|_{\mathcal H^1}\le C^2s^{-1-\ep}\| \bvec \Psi_0\|_{\mathcal H^1}\int_s^t r^{-1-\ep}(r-s)dr\le \frac{C^2s^{-2-2\ep}(t-s)^2}{2}\|\bvec\Psi_0\|_{\mathcal H^1}.
$$
By iterating this procedure, we have for any $n=1,2,...$
\begin{align*}
\|\bvec \Psi_{n+1}(t)-\bvec \Psi_n(t)\|_{\mathcal H^1}
\le \frac{C^{n+1}s^{(-1-\ep)(n+1)}|t-s|^{n+1}\|\bvec \Psi_0\|_{\mathcal H^1}}{(n+1)!}. 
\end{align*}
Hence, the standard argument shows that $\{\bvec \Psi_{n}\}$ is a Cauchy sequence in $C([a,b],\mathcal H^1(\R))$ for any $0<a<b<\infty$. Since $a$ and $b$ are arbitrarily,  \eqref{lemma_propagator_proof_1} thus admits a unique global solution $\bvec \Psi \in C((0,\infty);\mathcal H^1(\R))$ satisfying 
\begin{align}
\label{lemma_propagator_proof_2}\|\bvec \Psi(t)\|_{\mathcal H^1}\le e^{Cs^{-1-\ep}|t-s|}\|\bvec \Psi_0\|_{\mathcal H^1}\end{align}
with some $C>0$ depending on $\|\varphi\|_{H^1}$. The items (1), (2) and (3) easily follow from the  Duhamel formula \eqref{lemma_propagator_proof_1}. The item (4) follows by a direct computation. Indeed, if we denote by $\psi(t)$ the first component of $\mathcal U(t,s)\vec \psi_0$, then $\psi(t)$ solves \eqref{equation_psi_1}. Moreover, $\vec\psi(t)$ solves \eqref{linearized_equation} with $\vec\psi(s)=\vec\psi_0$. Thus, the uniqueness of the Cauchy problem implies $\mathcal U(t,s)\vec \psi_0=\vec \psi(t)$. 
\end{proof}

Using the above notations, \eqref{equation_v_2} can be written as the following nonlinear system
\begin{align}
\left(i\partial_t -\mathcal H(t)\right)(\vec v_*-\vec e_1)
=\mathcal J(\vec G[v_*]+\vec e_2),
\end{align}
where $\vec v_*=(v_*,\overline{v_*})^{\mathrm T}$, $\vec G[v_*]=(G[v_*],\overline{G[v_*]})^{\mathrm T}$ and
$$
G[v_*](t)=\sum_{j=1}^2\lambda_j t^{\sigma_j-2}G_{\sigma_j}[v_*(t)]
$$
with $G_{\sigma_j}[v_*]$ given by \eqref{G_sigma_j}, $\vec e_j=(e_j,\overline{e_j})^{\mathrm T}$, $e_1(t)=\mathcal R(t)v_\mathrm p(t)$ and 
\begin{align*}
e_2(t)&=\sum_{j=1}^2\lambda_j t^{\sigma_j-2}\left\{-\mathcal R(t)|v_\mathrm p(t)|^{2\sigma_j}v_\mathrm p(t)+(\sigma_j+1)|\varphi|^{2\sigma_j}e_1(t)+ \sigma_j |\varphi|^{2\sigma_j-2}v_{\mathrm p}(t)^2\overline{e_1(t)}\right\}. 
\end{align*}
With the condition $\|v_*(t)\|\to 0$ as $t\to+ 0$ at hand, we finally arrive at the integral equation: 
\begin{align}
\label{integral_equation}
\vec v_*(t)=\vec e_1(t)-\int_0^t \mathcal U(t,s)\left\{\vv{(iG)}[v_*](s)+\vv{(ie_2)}(s)\right\}ds,
\end{align}
where we used the formula $i\mathcal J\vec a=(ia,-i\overline a)^{\mathrm T}=\vv{(ia)}$. 

\subsection{Energy estimates for the linearized equation}
\label{subsection_energy}
The goal of the rest of the paper is to construct a unique global solution to \eqref{integral_equation}. To this end, we prove a key energy estimate for $\mathcal U(t,,s)\vec \psi_0$, which is the most important step in this paper. Thanks to the item (4) in Lemma \ref{lemma_propagator}, it is enough to deal with the solution to  \eqref{equation_psi_1}. We begin with the following energy identity: 

\begin{lemma}[Modified energy identity]
\label{lemma_energy_identity}
Suppose $\lambda_1,\lambda_2,\alpha\in \R$ and $\varphi\in H^1(\R)$. Define
\begin{align}
\nonumber
\widetilde{Q}_{\alpha}[\psi](t)
&=Q_\alpha[\psi](t)+\lambda_2\sigma_2 t^{\sigma_2-2}\||\varphi|^{\sigma_2-1}\Re[\overline{v_{\mathrm p}(t)}\psi(t)]\|^2\\
\label{widetilde_Q}
&=\frac{\|\partial_x  \psi(t)\|^2}{4}+t^{-\alpha} \|\psi(t)\|^2+\sum_{j=1}^2\lambda_j\sigma_j t^{\sigma_j-2}\||\varphi|^{\sigma_j-1}\Re[\overline{v_{\mathrm p}(t)}\psi(t)]\|^2. 
\end{align}
Then, for any $H^1$-solution $\psi(t)$ to \eqref{equation_psi_1} and  $t>0$, 
\begin{align}
\nonumber
\frac{d}{dt}\widetilde{Q}_{\alpha}[\psi]
\nonumber
&=
-\alpha t^{-\alpha-1} \|\psi\|^2
-\sum_{j=1}^2\lambda_j\sigma_j(2-\sigma_j) t^{\sigma_j-3}\||\varphi|^{\sigma_j-1}\Re[\overline{v_{\mathrm p}}\psi]\|^2\\
\nonumber
&\quad -4\sum_{j=1}^2\lambda_j\sigma_j t^{\sigma_j-2-\alpha}\<|\varphi|^{2\sigma_j-2}\Re[\overline{v_{\mathrm p}}\psi],\Im[\overline{v_{\mathrm p}}\psi]\>\\
\label{lemma_energy_identity_2}
&\quad 
-\sum_{j=1}^2\lambda_j\sigma_j  t^{\sigma_j-2}\Im\<\partial_x  \psi,|\varphi|^{2\sigma_j-2}\Re[\varphi\overline{\partial_x  \varphi}]\psi\>.
\end{align}
\end{lemma}

\begin{proof}
It is easy to see from \eqref{equation_psi_1} that
\begin{align}
\nonumber
\frac{d}{dt}\left(t^{-\alpha} \|\psi\|^2\right)
&=-\alpha t^{-\alpha-1} \|\psi\|^2+2\sum_{j=1}^2\lambda_j\sigma_j  t^{\sigma_j-2-\alpha}\Im\<|\varphi|^{2\sigma_j-2}v_{\mathrm p}^2\overline \psi,\psi\>\\
\label{equation_psi_3}
&=-\alpha t^{-\alpha-1} \|\psi\|^2+4\sum_{j=1}^2\lambda_j\sigma_j  t^{\sigma_j-2-\alpha}\<|\varphi|^{2\sigma_j-2}\Re[\overline{v_{\mathrm p}}\psi],\Im[\overline{v_{\mathrm p}}\psi]\>,
\end{align}
where we used $\Im[z^2]=2\Re z\Im z$. 
We next calculate both sides of the identity 
\begin{align}
\Re \<\text{LHS of }\eqref{equation_psi_1},\partial_t \psi\>=\Re \<\text{RHS of }\eqref{equation_psi_1},\partial_t \psi\>.
\label{equation_psi_4}
\end{align}
A direct computation yields that
\begin{align}
\label{equation_psi_4_1}
\Re\<i\partial_t \psi-H_0\psi,\partial_t \psi\>=-\frac14\frac{d}{dt}\left\|\partial_x  \psi\right\|^2,\quad
\Re\<|\varphi|^{2\sigma_j}\psi,\partial_t\psi\>=\frac12\frac{d}{dt}\left\||\varphi|^{\sigma_j} \psi\right\|^2.
\end{align}
Moreover, by using \eqref{v_p} and the expression \eqref{equation_psi_2}, we have
\begin{align}
\nonumber
&\Re\<|\varphi|^{2\sigma_j-2}v_{\mathrm p}\Re[\overline{v_\mathrm p}\psi],\partial_t \psi\>\\
\nonumber
&=\Re\<|\varphi|^{2\sigma_j-2}\Re[\overline{v_\mathrm p}\psi],\partial_t (\overline{v_{\mathrm p}}\psi)\>
-\Re\<|\varphi|^{2\sigma_j-2}\Re[\overline{v_\mathrm p}\psi],(\partial_t \overline{v_{\mathrm p}})\psi\>\\
\nonumber
&=\<|\varphi|^{2\sigma_j-2}\Re[\overline{v_\mathrm p}\psi],\partial_t \Re[\overline{v_{\mathrm p}}\psi]\>
-\sum_{k=1}^2\Re\<|\varphi|^{2\sigma_j-2}\Re[\overline{v_\mathrm p}\psi],i\lambda_k t^{\sigma_k-2}|\varphi|^{2\sigma_k}\overline{v_{\mathrm p}}\psi\>\\
\label{equation_psi_4_2}
&=\frac12\frac{d}{dt}\left\||\varphi|^{\sigma_j-1}\Re[\overline{v_\mathrm p}\psi]\right\|^2
+\sum_{k=1}^2\lambda_k  t^{\sigma_k-2}\<|\varphi|^{2(\sigma_j+\sigma_k)-2}\Re[\overline{v_\mathrm p}\psi],\Im[\overline{v_\mathrm p}\psi]\>.
\end{align}
To be precise, in order to make sense the quantity $\<i\partial_t \psi-H_0\psi,\partial_t\psi\>$, we first regard $\<\cdot,\cdot\>$ as the coupling $\<\cdot,\cdot\>_{H^{-1},H^1}$, replace $\partial_t \psi\in H^{-1}$ in the second entries of each three terms in the LHS of \eqref{equation_psi_4_1} and \eqref{equation_psi_4_2} by $(1+\ep H_0)^{-1}\partial_t\psi\in H^1$ and then take the limit $\ep\to +0$. This is possible since $|\varphi|^{2\sigma_j}\psi\in H^1,|\varphi|^{2\sigma_j-2}v_{\mathrm p}\Re[\overline{v_\mathrm p}\psi]\in H^1$, and $(1+\ep H_0)^{-1/2}$ commutes with $i\partial_t-H_0$. Plugging \eqref{equation_psi_4_1} and \eqref{equation_psi_4_2} into \eqref{equation_psi_4} implies
\begin{align}
\nonumber
-\frac{1}{4}\frac{d}{dt}\|\partial_x \psi\|^2&=\sum_{j=1}^2\lambda_j t^{\sigma_j-2}\frac{d}{dt}\left(\frac{\||\varphi|^{\sigma_j} \psi\|^2}{2}+\sigma_j \||\varphi|^{\sigma_j-1}\Re[\overline{v_{\mathrm p}}\psi]\|^2\right)\\
\label{equation_psi_5}
&\quad +2\sum_{j,k=1}^2\lambda_j\lambda_k \sigma_j  t^{\sigma_j+\sigma_k-4}\<|\varphi|^{2(\sigma_j+\sigma_k)-2}\Re[\overline{v_\mathrm p}\psi],\Im[\overline{v_\mathrm p}\psi]\>.
\end{align}
Similarly, calculating both sides of 
$$
\Im \<\text{LHS of }\eqref{equation_psi_1},|\varphi|^{2\sigma_j}\psi\>=\Im \<\text{RHS of }\eqref{equation_psi_1},|\varphi|^{2\sigma_j}\psi\>
$$
implies
\begin{align*}
&\frac12\frac{d}{dt}\||\varphi|^{\sigma_j} \psi\|^2-\Im\<\partial_x \psi,\sigma_j|\varphi|^{2\sigma_j-2}\Re[\overline{\varphi}\partial_x\varphi]\psi\>\\
&=-2\sum_{k=1}^2{\lambda_k\sigma_k}{t^{\sigma_k-2}}\<\Re[\overline{v_{\mathrm p}}\psi],|\varphi|^{2(\sigma_k+\sigma_j)-2}\Im[\overline{v_{\mathrm p}}\psi]\>.
\end{align*}
Hence the last term in \eqref{equation_psi_5} is rewritten as
\begin{align}
\nonumber
&2\sum_{j,k=1}^2\lambda_j\lambda_k \sigma_j  t^{\sigma_j+\sigma_k-4}\<|\varphi|^{2(\sigma_j+\sigma_k)-2}\Re[\overline{v_\mathrm p}\psi],\Im[\overline{v_\mathrm p}\psi]\>\\
\label{equation_psi_6}
&=\sum_{j=1}^2\lambda_j t^{\sigma_j-2}\left(-\frac{d}{dt}\frac{\left\||\varphi|^{\sigma_j} \psi\right\|^2}{2}
+\sigma_j \Im\<\partial_x \psi,|\varphi|^{2\sigma_j-2}\Re[\varphi\overline{\partial_x\varphi}]\psi\>\right).
\end{align}
Plugging \eqref{equation_psi_6} into \eqref{equation_psi_5} then implies
\begin{align*}
\nonumber
-\frac{1}{4}\frac{d}{dt}\|\partial_x \psi\|^2
&=\sum_{j=1}^2\lambda_j\sigma_j t^{\sigma_j-2}\left(\frac{d}{dt} \||\varphi|^{\sigma_j-1}\Re[\overline{v_{\mathrm p}}\psi]\|^2
+\Im\<\partial_x \psi,|\varphi|^{2\sigma_j-2}\Re[\varphi\overline{\partial_x \varphi}]\psi\>\right)\\
&=\sum_{j=1}^2\lambda_j\sigma_j\left( \frac{d}{dt}\left\{t^{\sigma_j-2}\||\varphi|^{\sigma_j-1}\Re[\overline{v_{\mathrm p}}\psi]\|^2\right\}+(2-\sigma_j)t^{\sigma_j-3}\||\varphi|^{\sigma_j-1}\Re[\overline{v_{\mathrm p}}\psi]\|^2\right)\\
&\quad 
+\sum_{j=1}^2\lambda_j \sigma_j t^{\sigma_j-2}\Im\<\partial_x \psi,|\varphi|^{2\sigma_j-2}\Re[\varphi\overline{\partial_x \varphi}]\psi\>,
\end{align*}
which is equivalent to
\begin{align}
\nonumber
&\frac{d}{dt}\left(\frac{\|\partial_x \psi\|^2}{4}+\sum_{j=1}^2\lambda_j\sigma_j t^{\sigma_j-2}\||\varphi|^{\sigma_j-1}\Re[\overline{v_{\mathrm p}}\psi]\|^2\right)\\
\nonumber
&=-\sum_{j=1}^2\lambda_j\sigma_j(2-\sigma_j)t^{\sigma_j-3}\||\varphi|^{\sigma_j-1}\Re[\overline{v_{\mathrm p}}\psi]\|^2\\
&\quad -\sum_{j=1}^2\lambda_j \sigma_j t^{\sigma_j-2}\Im\<\partial_x \psi,|\varphi|^{2\sigma_j-2}\Re[\varphi\overline{\partial_x \varphi}]\psi\>
\label{equation_psi_8}. 
\end{align}
\eqref{lemma_energy_identity_2} now follows from  \eqref{equation_psi_3} and \eqref{equation_psi_8}. 
\end{proof}

The above energy identity leads the following energy estimate: 

\begin{lemma}[Key energy estimate]
\label{lemma_energy_estimate}
Let $\lambda_1>0$, $\lambda_2\in \R$,  $0<\alpha<1$ and $\varphi\in H^1(\R)$. Then, there exists $C>0$ such that for any solution $\psi\in C((0,1],H^1(\R))$ to \eqref{equation_psi_1} and $0<s\le t\le 1$, \begin{align}
\label{lemma_energy_estimate_1}
Q_\alpha[\psi(t)](t)\le C Q_\alpha[\psi(s)](s). 
\end{align}
\end{lemma}

\begin{proof}
Note that $\mathcal U(t,s)\vec \psi_0=(\psi,\overline\psi)^{\mathrm T}$ with $\psi_0=\psi(s)$ and thus $Q_\alpha[\mathcal U(t,s)\vec \psi_0](t)=2Q_\alpha[\psi(t)](t)$ by Lemma \ref{lemma_propagator} (4) and that $Q_\alpha[f](t)\sim \|f\|_{H^1}$ for each $t>0$ since $\lambda_1>0$. Thus, \eqref{lemma_propagator_proof_2} implies that for any $t_0>0$ there exists $C>0$ such that
$$
Q_\alpha[\psi(t)](t)\le C\|\mathcal U(t,s)\vec \psi_0\|_{\mathcal H^1}\le C\|\vec \psi_0\|_{\mathcal H^1}\le C\|\psi(s)\|_{H^1}\le CQ_\alpha[\psi(s)](s)
$$
for $t_0<s<t\le1$. It thus suffices to prove the lemma for sufficiently small $s\le t$. We set for short 
$
D_j[\psi](t)=\sigma_j\||\varphi|^{\sigma_j-1}\Re[\overline{v_{\mathrm p}(t)}\psi(t)]\|^2
$. 
Since $\sigma_1=1$ and $\sigma_2=\sigma>1$, we have for sufficiently small $t_0>0$ and any $t$ satisfying $0<t\le t_0$, 
\begin{align}
\label{energy_proof_0}
|\lambda_2|(2-\sigma_2)t^{\sigma-2}D_2[\psi](t)
\le |\lambda_2|\sigma_2(2-\sigma_2)  t^{\sigma-1} t^{-1}\|\varphi\|_{L^\infty}^{\sigma_2-1}D_1[\psi](t)<\frac{\lambda_1t^{-1}D_1[\psi](t)}{4}.
\end{align}
In particular, $C^{-1} Q_\alpha[\psi](t)\le \widetilde{Q}_{\alpha}[\psi](t)\le CQ_\alpha[\psi](t)$ and $\widetilde{Q}_{\alpha}[\psi](t)-\|\partial_x\psi(t)\|^2/4$ is positive if $t>0$ is small enough. To deal with the third term of the RHS of \eqref{lemma_energy_identity_2}, since $\alpha<\sigma_j$ and thus $t^{\sigma_j-2-\alpha}\le  t^{-1-\alpha}=t^{\frac{1-\alpha}{2}}t^{-\frac{\alpha+1}{2}}t^{-1}$ for $j=1,2$ and $0<t\le1$, we obtain by Young's inequality,   
\begin{align}
\nonumber
&\sum_{j=1}^2\left|2 \lambda_j t^{\sigma_j-2-\alpha}\<|\varphi|^{2\sigma_j-2}\Re[\overline{v_{\mathrm p}(t)}\psi(t)],\Im[\overline{v_{\mathrm p}(t)}\psi(t)]\>\right|\\
\nonumber
&\le C t^{\frac{1-\alpha}{2}}\cdot t^{-\frac{\alpha+1}{2}} \|\psi(t)\| \cdot t^{-1}\|\Re[\overline{v_{\mathrm p}(t)}\psi(t)]\|\\
\label{energy_proof_1}
&\le \frac{\alpha t^{-\alpha-1}\|\psi(t)\|^2}{4}+\frac{\lambda_1  t^{-2}D_1[\psi](t)}{4}
\end{align}
for $0<t\le t_0$ by taking $t_0$ further small if necessary. For the last term of \eqref{lemma_energy_identity_2}, we use Gagliardo--Nirenberg's inequality (see \cite{Cazenave}): 
\begin{align}
\label{GN}
\|f\|_{L^\infty}\lesssim \|\partial_x  f\|^{1/2} \|f\|^{1/2},\quad f\in H^1(\R), 
\end{align}
which, combined with Young's inequality, implies
$$
\|f\|_{L^\infty}\le C_1 \left(\delta\|\partial_xf\|+\delta^{-1}\|f\|\right)
$$
with some $C_1>0$ and for all $\delta>0$. Therefore, for any $\delta,\rho>0$, 
\begin{align*}
t^{-1}|\<\partial_x  \psi(t),\Re[\varphi\overline{\partial_x  \varphi}]\psi(t)\>|
&\le t^{-1}\|\partial_x  \psi(t)\| \|\varphi\|_{L^\infty}\|\partial_x  \varphi\|\|\psi(t)\|_{L^\infty}\\
&\le C_1 t^{-1}\|\partial_x  \psi(t)\|\left(\delta \|\partial_x  \psi(t)\|+\delta^{-1}\|\psi(t)\|\right)\\
&= C_1t^{-1}\delta \|\partial_x  \psi(t)\|^2+C_1\delta^{-1}t^{(\alpha-1)/2} \|\partial_x  \psi(t)\|\cdot t^{-(\alpha+1)/2}\|\psi(t)\|\\
&\le C_1  t^{-1}\delta \|\partial_x  \psi(t)\|^2+C_1\left(\rho^{-1}\delta^{-2}t^{\alpha-1} \|\partial_x  \psi(t)\|^2+\rho t^{-\alpha-1}\|\psi(t)\|^2\right)
\end{align*}
Let $\rho=(4C_1\lambda_1)^{-1}\alpha$. Then
there exists $C>0$ independent of $t$ such that
\begin{align}
\nonumber
\left|\lambda_1 t^{-1}\Im\<\partial_x  \psi(t),\Re[\varphi\overline{\partial_x  \varphi}]\psi(t)\>\right|
&\le C\left(t^{-1}\delta+\delta^{-2}t^{\alpha-1} \right)\|\partial_x  \psi(t)\|^2+\frac{\alpha t^{-\alpha-1}\|\psi(t)\|^2}{4}\\
\label{energy_proof_2}
&\le \frac{Ct^{\alpha/3-1}\|\partial_x  \psi(t)\|^2}{8}+\frac{\alpha t^{-\alpha-1}\|\psi(t)\|^2}{4},
\end{align}
where we take $\delta=t^{\alpha/3}$ so that 
$
t^{-1}\delta=\delta^{-2}t^{\alpha-1} =t^{\alpha/3-1}
$. Since $\sigma_2>1$, we similarly obtain
\begin{align}
\left|\lambda_2 t^{\sigma_2-2}\Im\<\partial_x  \psi(t),|\varphi|^{2\sigma_2-2}\Re[\varphi\overline{\partial_x  \varphi}]\psi(t)\>\right|
\label{energy_proof_3}
&\le \frac{Ct^{\alpha/3-1}\|\partial_x  \psi(t)\|^2}{8}+\frac{\alpha t^{-\alpha-1}\|\psi(t)\|^2}{4},
\end{align}
It follows from \eqref{energy_proof_0}, \eqref{energy_proof_1}, \eqref{energy_proof_2}, \eqref{energy_proof_3} and the modified energy identity \eqref{lemma_energy_identity_2} that
\begin{align*}
\frac{d}{dt}\widetilde{Q}_{\alpha}[\psi](t)
&\le -\alpha t^{-\alpha-1}\|\psi(t)\|^2-\lambda_1 t^{-2}D_1[\psi](t)\\
&\quad +\frac{\lambda_1 t^{-2}D_1[\psi](t)}{2}+\frac{Ct^{\alpha/3-1}\|\partial_x  \psi(t)\|^2}{4}+\frac{3\alpha t^{-\alpha-1}\|\psi(t)\|^2}{4}\\
&\le Ct^{\alpha/3-1}\widetilde{Q}_{\alpha}[\psi](t),
\end{align*}
where we have used the positivity of $\widetilde{Q}_{\alpha}[\psi]-\|\partial_x\psi\|^2/4$. This inequality implies for $0<s\le t\le t_0$, 
$$
\widetilde{Q}_{\alpha}[\psi](t)\le \exp\left(C\int_s^t r^{\alpha/3-1}dr\right)\widetilde{Q}_{\alpha}[\psi](s)\lesssim \widetilde{Q}_{\alpha}[\psi](s)
$$ and \eqref{lemma_energy_estimate_1} follows since $Q_\alpha[\psi](t)\sim \widetilde{Q}_{\alpha}[\psi](t)$ for $0<t\le t_0$. 
\end{proof}

By this lemma and Lemma \ref{lemma_propagator} (4), we obtain the following estimate for $\mathcal U(t,s)$: 
\begin{corollary}
\label{corollary_energy_estimate}
Under the same condition in Lemma \ref{lemma_energy_estimate}, for all $\psi_0\in H^1(\R)$, 
$$
Q_\alpha[\mathcal U(t,s)\vec \psi_0](t)\lesssim Q_\alpha[\psi_0](s)
$$
uniformly in $0<s\le t\le 1$, where $
Q_\alpha[(f_1,f_2)^\mathrm T](t):=Q_\alpha[f_1(t)]+Q_\alpha[\overline{f_2(t)}]
$ for $(f_1,f_2)^{\mathrm T}\in \mathcal H^1(\R)$. 
\end{corollary}

\begin{proof}
Lemma \ref{lemma_propagator} (4) implies $\mathcal U(t,s)\vec \psi_0=(\psi,\overline\psi)^{\mathrm T}$ with the solution $\psi$ to \eqref{equation_psi_1} with $\psi(s)=\psi_0$. Thus, $Q_\alpha[\mathcal U(t,s)\vec \psi_0](t)=2Q_\alpha[\psi(t)](t)$ and the desired estimate follows from Lemma \ref{lemma_energy_estimate}. 
\end{proof}

\section{Proof of Theorem \ref{theorem_1}}
\label{section_proof_Theorem_1}

Using the materials prepared in the previous section, we here prove Theorem \ref{theorem_1}. Let $\Phi[\vec v_*](t)$ be the RHS of the integral equation \eqref{integral_equation}. We shall show that $\Phi$ is a contraction on the following complete metric space $\mathcal X(T,\alpha,\beta,M)$ for sufficiently small $T>0$: 
\begin{equation}
\label{X}
\begin{aligned}
\mathcal X&=\mathcal X(T,\alpha,\beta,M)=\{\vec f=(f,\overline f)^{\mathrm T}\in C((0,T];\mathcal H^1(\R))\ |\ \|\vec f\|_\mathcal X\le M\}, \\
\|\vec f\|_\mathcal X&=\sup_{0<t\le T} t^{-\beta}\left(\|\partial_x f\|+t^{-\alpha/2}\|f\|\right),\quad
d_\mathcal X(\vec f,\vec g)=\|\vec f-\vec g\|_\mathcal X. 
\end{aligned}
\end{equation}
Since $\|\partial_x f\|+t^{-\alpha/2}\|f\|\lesssim \sqrt{Q_\alpha[f](t)}$ with $Q_\alpha[f]$ defined by \eqref{Q_1}, Corollary \ref{corollary_energy_estimate} and \eqref{integral_equation} imply
\begin{align}
\|\Phi[\vec v_*]\|_{\mathcal X}
\label{Phi}
&\lesssim \|\vec e_1\|_{\mathcal X}+\sup_{0<t\le T} t^{-\beta}\int_0^t \left(\sqrt{Q_\alpha\big[iG[v_*]\big](s)}+\sqrt{Q_\alpha\big[ie_2\big](s)}\right)ds,\\
\|\Phi[\vec v_1]-\Phi[\vec v_2]\|_{\mathcal X}
\label{Phi-Phi}
&\lesssim \sup_{0<t\le T} t^{-\beta}\int_0^t \sqrt{Q_\alpha\big[iG[v_1]-iG[v_2]\big](s)}ds
\end{align}
for $\vec v_*,\vec v_{1},\vec v_{2}\in \mathcal X(T,\alpha,\beta,M)$. 
Let us begin to deal with the nonlinear term $G$: 

\begin{proposition}
\label{proposition_G}
Let $\varphi\in H^1(\R)$ and $\alpha>0$. Then, for any $0<t\le T$ and $\vec v_*,\vec v_1,\vec v_2\in \mathcal X(T,\alpha,\beta,M)$, 
\begin{align}
\sqrt{Q_\alpha\big[iG[v_*]\big](t)}
\label{proposition_G_1}
&\lesssim t^{2\beta+\min\{\alpha/4,3\alpha/4-1/2\}-1}\<M\>^{2\sigma+1},\\
\sqrt{Q_\alpha\big[iG[v_1]-iG[v_2]\big](t)}
\label{proposition_G_2}
&\lesssim t^{2\beta+\min\{\alpha/4,3\alpha/4-1/2\}-1}\<M\>^{2\sigma}
\|\vec v_1-\vec v_2\|_{\mathcal X}.
\end{align}
\end{proposition}

\begin{proof}
Recall that  $G[v_*]=\lambda_1t^{\sigma_1-2}G_{\sigma_1}+\lambda_2t^{\sigma_2-2}G_{\sigma_2}
$ with $\sigma_1=1$, $\sigma_2=\sigma$ and $G_{\sigma_j}=G_{\sigma_j}[v_*]$ defined by \eqref{G_sigma_j}. Let us deal with $G_\sigma$ for general $\sigma\ge1$ defined by
\begin{align}
\label{G_proof_1}
G_\sigma[v_*]:=(\sigma+1)v_{*}\int_0^1\left(|v_\theta|^{2\sigma}-|v_{\mathrm p}|^{2\sigma}\right)d\theta+\sigma\overline{v_{*}}\int_0^1\left(|v_\theta|^{2\sigma-2}v_\theta^2-|v_{\mathrm p}|^{2\sigma-2}v_{\mathrm p}^2\right)d\theta. 
\end{align}
In what follows, we frequently use the inequalities: 
\begin{align}
\left||v_\theta|^{2\sigma}-|v_{\mathrm p}|^{2\sigma}\right|+\left||v_\theta|^{2\sigma-2}v_\theta^2-|v_{\mathrm p}|^{2\sigma-2}v_{\mathrm p}^2\right|
\label{G_proof_4_1}
&\lesssim |v_*|^{2\sigma}+|\varphi|^{2\sigma-1}|v_*|,\\
\label{G_proof_4_2}
\left||v_\theta|^{2\sigma-1-n}v_\theta^n-|v_{\mathrm p}|^{2\sigma-1-n}v_{\mathrm p}^n\right|
&\lesssim |v_*|^{2\sigma-1}+|\varphi|^{2\sigma-2}|v_*|,\quad n=0,1,2,...,
\end{align}
which follow from the following elementary inequalities (see \cite[Lemma 2.4]{Ginibre_Velo_1989}):
\begin{align*}
\left||z_1|^{p-n}z_1^n-|z_2|^{p-n}z_2^n\right|\lesssim |z_1-z_2|^p+(|z_1|^{p-1}+|z_2|^{p-1})|z_1-z_2|,\quad p\ge1,\ z_1,z_2\in \C.  
\end{align*}

Now we show \eqref{proposition_G_1}. With the definition \eqref{Q_1} of $Q_\alpha$ at hand, we need to estimate $\|G\|$, $\|\Re[\overline{v_{\mathrm p}}iG]\|$ and $\|\partial_x G\|$. For the latter two norms, plugging \eqref{G_proof_4_1} into \eqref{G_proof_1} implies
\begin{align}
\label{G_proof_5}
\|\Re[\overline{v_{\mathrm p}}iG_\sigma[v_*]]\|\lesssim \|G_\sigma[v_*]\|
\lesssim \||v_*|^{2\sigma+1}\|+\||\varphi|^{2\sigma-1}|v_*|^2\|
\lesssim \left(\|v_*\|_{L^\infty}^{2\sigma}+\| v_*\|_{L^\infty}\right)\|v_*\|. 
\end{align}
To deal with the term $\partial_x G[v_*]$, calculating
\begin{align*}
\partial_x \left(|v_\theta|^{2\sigma}-|v_{\mathrm p}|^{2\sigma}\right)
&=\sigma |v_\theta|^{2\sigma-2}\left(v_\theta\overline{\partial_x v_\theta}+\overline{v_\theta}\partial_x v_\theta\right)-\sigma |v_{\mathrm p}|^{2\sigma-2}\left(v_{\mathrm p}\overline{\partial_x v_{\mathrm p}}+\overline{v_{\mathrm p}}\partial_x v_{\mathrm p}\right),\\
&=2\sigma \Re\left[|v_{\theta}|^{2\sigma-2}v_\theta-|v_{\mathrm p}|^{2\sigma-2}v_{\mathrm p}\right]\overline{\partial_x v_\theta}+2\sigma \Re\big[|v_{\mathrm p}|^{2\sigma-2}v_{\mathrm p}\overline{\partial_x(v_\theta-v_{\mathrm p})}\big]
\end{align*}
and 
\begin{align*}
&\partial_x \left(|v_\theta|^{2\sigma-2}v_\theta^2-|v_{\mathrm p}|^{2\sigma-2}v_{\mathrm p}^2\right)\\
&=(\sigma+1)\left\{|v_\theta|^{2\sigma-2}v_\theta\partial_x v_\theta-|v_{\mathrm p}|^{2\sigma-2}v_{\mathrm p}\partial_x v_{\mathrm p}\right\}
+(\sigma-1)\left\{|v_\theta|^{2\sigma-4}v_\theta^3\overline{\partial_x v_\theta}-|v_{\mathrm p}|^{2\sigma-4}v_{\mathrm p}^3\overline{\partial_x v_{\mathrm p}}\right\}\\
&=(\sigma+1)\left(|v_\theta|^{2\sigma-2}v_\theta-|v_{\mathrm p}|^{2\sigma-2}v_{\mathrm p}\right)\partial_x v_\theta
+\theta (\sigma+1)|v_{\mathrm p}|^{2\sigma-2}v_{\mathrm p}\partial_xv_*\\
&\quad +(\sigma-1)\left(|v_\theta|^{2\sigma-4}v_\theta^3-|v_{\mathrm p}|^{2\sigma-4}v_{\mathrm p}^3\right)\overline{\partial_x v_\theta}+\theta (\sigma-1)|v_{\mathrm p}|^{2\sigma-4}v_{\mathrm p}^3\overline{\partial_xv_*},
\end{align*}
we similarly obtain by using \eqref{G_proof_4_1} and \eqref{G_proof_4_2} that
\begin{align*}
&\left|\partial_x\{v_* (|v_\theta|^{2\sigma}-|v_{\mathrm p}|^{2\sigma})\}\right|+\left|\partial_x\{\overline{v_*}(|v_\theta|^{2\sigma-2}v_\theta^2-|v_{\mathrm p}|^{2\sigma-2}v_{\mathrm p}^2)\}\right|\\
&\le |\partial_x v_*|\left(\left||v_\theta|^{2\sigma}-|v_{\mathrm p}|^{2\sigma}\right|+\left||v_\theta|^{2\sigma-2}v_\theta^2-|v_{\mathrm p}|^{2\sigma-2}v_{\mathrm p}^2\right|\right)\\
&\quad +|v_*|\left\{\left|\partial_x \left(|v_\theta|^{2\sigma}-|v_{\mathrm p}|^{2\sigma}\right)\right|+\left|\partial_x \left(|v_\theta|^{2\sigma-2}v_\theta^2-|v_{\mathrm p}|^{2\sigma-2}v_{\mathrm p}^2\right)\right|\right\}\\
&\lesssim |\partial_x v_*|\left(|v_*|^{2\sigma}+|\varphi|^{2\sigma-1}|v_*|\right)
+|v_*|\left(|v_*|^{2\sigma-1}+|\varphi|^{2\sigma-2}|v_*|\right)|\partial_x v_\theta|+|v_*||\varphi|^{2\sigma-1}|\partial_x v_*|
\\
&\lesssim \left(|v_*|^{2\sigma}+|\varphi|^{2\sigma-2}|v_*|^2+|\varphi|^{2\sigma-1}|v_*|\right)|\partial_xv_*|
+|\partial_x v_{\mathrm p}|\left(|v_*|^{2\sigma}+|\varphi|^{2\sigma-2}|v_*|^2\right).
\end{align*}
Since $v_{\mathrm p}(t)=\overline{\varphi}e^{-i\lambda_1 |\varphi|^2\log t-i\lambda_2|\varphi|^{2\sigma}t^{\sigma-1}/(\sigma-1)}$, this inequality yields
\begin{align*}
|\partial_x G_\sigma[v_*]|
&\lesssim \left\{|v_*|^{2\sigma}+|\varphi|^{2\sigma-1}|v_*|+|\varphi|^{2\sigma-2}|v_*|^2\right\}|\partial_xv_*|\\
&\quad +\left(1+|\varphi|^{2\sigma}+| \log t| |\varphi|^2\right)|\partial_x\varphi|\left(|v_*|^{2\sigma}+|\varphi|^{2\sigma-2}|v_*|^2\right).
\end{align*}
This implies
\begin{align}
\nonumber
\|\partial_xG_\sigma[v_*]\|
\nonumber
&\lesssim \left(\|v_*\|_{L^\infty}^{2\sigma}+\| v_*\|_{L^\infty}^2+\|v_*\|_{L^\infty}\right)\|\partial_x v_*\|\\
\nonumber
&\quad +\|\partial_x\varphi\|\{(1+\|\varphi\|_{L^\infty}^{2\sigma})\|v_*\|_{L^\infty}^{2\sigma}+|\log t| \|\varphi v_*\|_{L^\infty}^2\|v_*\|_{L^\infty}^{2\sigma-2}\}\\
\nonumber
&\quad +\|\partial_x\varphi\|\{(1+\|\varphi\|_{L^\infty}^{2\sigma})\|\varphi\|_{L^\infty}^{2\sigma-2}\|v_*\|_{L^\infty}^{2}+ | \log t | \|\varphi v_*\|_{L^\infty}^2\|\varphi\|_{L^\infty}^{2\sigma-2}\}\\
\nonumber
&\lesssim \left(\|v_*\|_{L^\infty}^{2\sigma}+\|v_*\|_{L^\infty}^2+\| v_*\|_{L^\infty}\right)\|\partial_x v_*\|
+\|v_*\|_{L^\infty}^{2\sigma}+\|v_*\|_{L^\infty}^{2}\\
\label{G_proof_7}
&\quad +|\log t | \| v_*\|_{L^\infty}^{2}(1+\|v_*\|_{L^\infty}^{2\sigma-2}). 
\end{align}
To estimate the RHS of \eqref{G_proof_5} and \eqref{G_proof_7}, we use the following three bounds: 
\begin{align}
\label{G_proof_2}
\|v_*\|\le t^{\beta+\alpha/2}M,\quad \|\partial_x  v_*\|\le t^\beta M,\quad \|v_{*} \|_{L^{\infty}} 
\lesssim \| \partial_x   v_{\ast} \|^{1/2} \| v_{\ast} \|^{1/2}\lesssim t^{\beta+\alpha/4}M. 
\end{align}
It follows from \eqref{G_proof_5}--\eqref{G_proof_2} that
\begin{align*}
t^{-\alpha/2}\|G_\sigma[v_*]\|
&\lesssim t^{(2\sigma+1)\beta+\sigma \alpha/2}M^{2\sigma+1}+t^{2\beta+\alpha/4}M^2
\lesssim t^{2\beta+\alpha/4}\<M\>^{2\sigma+1},\\
t^{-1/2}\|\Re[\overline{v_{\mathrm p}}  G_\sigma[v_*]]\|
&\lesssim t^{2\beta+3\alpha/4-1/2}\<M\>^{2\sigma+1},\\
\|\partial_x  G_\sigma[v_*]\|
\nonumber
&\lesssim t^{(2\sigma+1)\beta+\sigma\alpha/2}M^{2\sigma+1}
+t^{3\beta+\alpha/2}M^3
+t^{2\beta+\alpha/4}M^2
+t^{2\sigma(\beta+\alpha/4)}M^{2\sigma}\\
&\quad +t^{2\beta+\alpha/2}M^2
+|\log t|t^{2\beta+\alpha/2}M^{2}
+|\log t|t^{2\sigma\beta+\sigma\alpha/2}M^{2\sigma}\\
&\lesssim t^{2\beta+\alpha/4}\<M\>^{2\sigma+1},
\end{align*}
and \eqref{proposition_G_1} follows. 

Next we prove \eqref{proposition_G_2}. Setting $v_{j\theta}=v_{\mathrm p}+\theta v_j$ for short, we write
\begin{align*}
&G_\sigma[v_1]-G_\sigma[v_2]
=(\sigma+1)\int_0^1\left\{v_1\left(|v_{1\theta}|^{2\sigma}-|v_{\mathrm p}|^{2\sigma}\right)-v_2\left(|v_{2\theta}|^{2\sigma}-|v_{\mathrm p}|^{2\sigma}\right)\right\}d\theta\\
&\quad +\sigma \int_0^1\left\{\overline {v_1}\left(|v_{1\theta}|^{2\sigma-2}v_{1\theta}^2-|v_{\mathrm p}|^{2\sigma-2}v_{\mathrm p}^2\right)-\overline{v_2}\left(|v_{2\theta}|^{2\sigma-2}v_{2\theta}^2-|v_{\mathrm p}|^{2\sigma-2}v_{\mathrm p}^2\right)\right\}d\theta.
\end{align*}
We shall only deal with the second integral which we denote by $I_\sigma[v_1,v_2]$ for short, the proof for the first one being analogous. By \eqref{G_proof_4_1}, the integrand of $I_\sigma[v_1,v_2]$ is estimated as
\begin{align*}
&\left|\overline{v_1}\left(|v_{1\theta}|^{2\sigma-2}v_{1\theta}^2-|v_{\mathrm p}|^{2\sigma-2}v_{\mathrm p}^2\right)
-\overline{v_2}\left(|v_{2\theta}|^{2\sigma-2}v_{2\theta}^2-|v_{\mathrm p}|^{2\sigma-2}v_{\mathrm p}^2\right)\right|\\
&\lesssim |v_1-v_2|(|v_1|^{2\sigma}+|\varphi|^{2\sigma-1}|v_1|)
+|v_2|\left\{|v_1-v_2|^{2\sigma}+(|v_1|^{2\sigma-1}+|v_2|^{2\sigma-1}+|\varphi|^{2\sigma-1})|v_1-v_2|\right\}\\
&\lesssim \sum_{j=1}^2\left(|v_j|^{2\sigma}+|v_j|\right)|v_1-v_2|. 
\end{align*}
Hence, it follows from \eqref{G_proof_2} that 
\begin{align}
\nonumber
&t^{-\alpha/2}\|I_\sigma[v_1,v_2]\|+t^{-1/2}\|\Re[\overline{v_{\mathrm p}} iI_\sigma[v_1,v_2]]\|\\
\nonumber
& \lesssim t^{\beta+\alpha/2-1/2} \sum_{j=1,2}  \left( \| v_j \|_{L^{\infty}}^{2 \sigma}  + \|  v_j  \| _{L^{\infty}}  \right)  (1+t^{})\| v_1- v_2\|_{\mathcal X}\\
\label{G_proof_11}
&\lesssim t^{2\beta+3\alpha/4-1/2}\<M\>^{2\sigma}
\| v_1- v_2\|_{\mathcal X}.
\end{align}
Moreover, the derivative of the integrand can be written as
\begin{align*}
&\partial_x \left\{\overline {v_1}\left(|v_{1\theta}|^{2\sigma-2}v_{1\theta}^2-|v_{\mathrm p}|^{2\sigma-2}v_{\mathrm p}^2\right)-\overline{v_2}\left(|v_{2\theta}|^{2\sigma-2}v_{2\theta}^2-|v_{\mathrm p}|^{2\sigma-2}v_{\mathrm p}^2\right)\right\}\\
&=\overline{(\partial_x v_1-\partial_x v_2)}\left(|v_{1\theta}|^{2\sigma-2}v_{1\theta}^2-|v_{\mathrm p}|^{2\sigma-2}v_{\mathrm p}^2\right)+\overline {\partial_x v_2}\left(|v_{1\theta}|^{2\sigma-2}v_{1\theta}^2-|v_{2\theta}|^{2\sigma-2}v_{2\theta}^2\right)\\
&\quad
+(\sigma+1)\overline{v_1}\left(|v_{1\theta}|^{2\sigma-2}v_{1\theta}-|v_{\mathrm p}|^{2\sigma-2}v_{\mathrm p}\right)\partial_x v_{1\theta}
+(\sigma+1)\overline{v_1}|v_{\mathrm p}|^{2\sigma-2}v_{\mathrm p}(\partial_x v_{1\theta}-\partial_x v_{\mathrm p})\\
&\quad 
-(\sigma+1)\overline{v_2}\left(|v_{2\theta}|^{2\sigma-2}v_{2\theta}-|v_{\mathrm p}|^{2\sigma-2}v_{\mathrm p}\right)\partial_x v_{2\theta}
-(\sigma+1)\overline{v_2}|v_{\mathrm p}|^{2\sigma-2}v_{\mathrm p}(\partial_x v_{2\theta}-\partial_x v_{\mathrm p})\\
&\quad 
+(\sigma-1)\overline{v_1}\left(|v_{1\theta}|^{2\sigma-4}v_{1\theta}^3-|v_{\mathrm p}|^{2\sigma-4}v_{\mathrm p}^3\right)\overline{\partial_x v_{1\theta}}
+(\sigma-1)\overline{v_1}|v_{\mathrm p}|^{2\sigma-4}v_{\mathrm p}^3\overline{(\partial_x v_{1\theta}-\partial_x v_{\mathrm p})}\\
&\quad -(\sigma-1)\overline{v_2}\left(|v_{2\theta}|^{2\sigma-4}v_{2\theta}^3-|v_{\mathrm p}|^{2\sigma-4}v_{\mathrm p}^3\right)\overline{\partial_x v_{2\theta}}
-(\sigma-1)\overline{v_2}|v_{\mathrm p}|^{2\sigma-4}v_{\mathrm p}^3\overline{(\partial_x v_{2\theta}-\partial_x v_{\mathrm p})}\\
&=\overline{(\partial_x v_1-\partial_x v_2)}\left(|v_{1\theta}|^{2\sigma-2}v_{1\theta}^2-|v_{\mathrm p}|^{2\sigma-2}v_{\mathrm p}^2\right)
+\overline {\partial_x v_2}\left(|v_{1\theta}|^{2\sigma-2}v_{1\theta}^2-|v_{2\theta}|^{2\sigma-2}v_{2\theta}^2\right)\\
&\quad
+(\sigma+1)\left(\overline{v_1}\partial_x v_{1\theta}-\overline{v_2}\partial_x v_{2\theta}\right)\left(|v_{1\theta}|^{2\sigma-2}v_{1\theta}-|v_{\mathrm p}|^{2\sigma-2}v_{\mathrm p}\right)\\
&\quad
+(\sigma+1)\overline{v_2}\partial_x v_{2\theta}\left(|v_{1\theta}|^{2\sigma-2}v_{1\theta}-|v_{2 \theta}|^{2\sigma-2}v_{2 \theta}\right)
+\theta (\sigma+1)(\overline{v_1}\partial_xv_1-\overline{v_2}\partial_xv_2)|v_{\mathrm p}|^{2\sigma-2}v_{\mathrm p}\\
&\quad
+(\sigma-1)\overline{(v_1\partial_x v_{1\theta}-v_2\partial_x v_{2\theta})}\left(|v_{1\theta}|^{2\sigma-4}v_{1\theta}^3-|v_{\mathrm p}|^{2\sigma-4}v_{\mathrm p}^3\right)\\
&\quad
+(\sigma-1)\overline{v_2\partial_x v_{2\theta}}\left(|v_{1\theta}|^{2\sigma-4}v_{1\theta}^3-|v_{2 \theta}|^{2\sigma-4}v_{2 \theta}^3\right)
+\theta(\sigma-1)(\overline{v_1\partial_xv_1}-\overline{v_2\partial_xv_2})|v_{\mathrm p}|^{2\sigma-4}v_{\mathrm p}^3. 
\end{align*}
Applying again \eqref{G_proof_4_1} and \eqref{G_proof_2}, we  obtain
\begin{align}
\nonumber
&\|(\partial_x v_1-\partial_x v_2)(|v_{1\theta}|^{2\sigma-2}v_{1\theta}^2-|v_{\mathrm p}|^{2\sigma-2}v_{\mathrm p}^2)\|\\
\label{G_proof_12}
& \lesssim \|\partial_x(v_1-v_2)\|\left(\|v_1\|_{L^\infty}^{2\sigma}+\|v_1\|_{L^\infty}\right)
\lesssim t^{2\beta+\alpha/4}\<M\>^{2\sigma}\|v_1-v_2\|_{\mathcal X},\\
\nonumber
&\| {\partial_x v_2}(|v_{1\theta}|^{2\sigma-2}v_{1\theta}^2-|v_{2\theta}|^{2\sigma-2}v_{2\theta}^2)\|\\
\nonumber
& \lesssim \|\partial_x v_2\|\left\{\|v_1-v_2\|_{L^\infty}^{2\sigma}+\left(\|v_1\|_{L^\infty}^{2\sigma-1}+\|v_2\|_{L^\infty}^{2\sigma-1}\right)\|v_1-v_2\|_{L^\infty}\right\}\\
\label{G_proof_13}
& \lesssim t^{2\sigma \beta+\sigma \alpha/2}\<M\>^{2\sigma}\|v_1-v_2\|_{\mathcal X}. 
\end{align}
Similarly, using \eqref{G_proof_4_2} instead of \eqref{G_proof_4_1}, we observe for $n=1,3$ that
\begin{align*}
&\left|\left({v_1}\partial_x v_{1\theta}-{v_2}\partial_x v_{2\theta}\right)\left(|v_{1\theta}|^{2\sigma-1-n}v_{1\theta}^n-|v_{\mathrm p}|^{2\sigma-1-n}v_{\mathrm p}^n\right)\right|\\
&\lesssim \left\{|v_1-v_2|\left(|\partial_xv_1|+|\partial_xv_{\mathrm p}|\right)+|v_2||\partial_x(v_1-v_2)|\right\}\left(|v_1|^{2\sigma-1}+|\varphi|^{2\sigma-2}|v_1|\right),
\end{align*}
which, combined with \eqref{G_proof_2}, yields
\begin{align}
\nonumber
&\|\left(\overline{v_1}\partial_x v_{1\theta}-\overline{v_2}\partial_x v_{2\theta}\right)(|v_{1\theta}|^{2\sigma-2}v_{1\theta}-|v_{\mathrm p}|^{2\sigma-2}v_{\mathrm p})\|\\
\nonumber
&\quad +\|{(v_1\partial_x v_{1\theta}-v_2\partial_x v_{2\theta})}(|v_{1\theta}|^{2\sigma-4}v_{1\theta}^3-|v_{\mathrm p}|^{2\sigma-4}v_{\mathrm p}^3)\|
\\
\nonumber
&\lesssim \|v_1-v_2\|_{L^\infty}\left\{(\|\partial_xv_1\|+1)\left(\|v_1\|_{L^\infty}^{2\sigma-1}+\|v_1\|_{L^\infty}\right)+|\log t | \|v_1\|_{L^\infty}^2(\|v_1\|_{L^\infty}^{2\sigma-2}+1)\right\}\\
\label{G_proof_14}
&\lesssim t^{2\beta+\alpha /4}\<M\>^{2\sigma}\|v_1-v_2\|_{\mathcal X},
\end{align}
and, similarly
\begin{align}
\nonumber
&\|(\overline{v_1}\partial_xv_1-\overline{v_2}\partial_xv_2)|v_{\mathrm p}|^{2\sigma-2}v_{\mathrm p}\|
+\|(\overline{v_1\partial_xv_1}-\overline{v_2\partial_xv_2})|v_{\mathrm p}|^{2\sigma-4}v_{\mathrm p}^3\|\\
\label{G_proof_15}
&\lesssim t^{2\beta+\alpha /4}\<M\>^{2\sigma}\|v_1-v_2\|_{\mathcal X}. 
\end{align}
Finally, since $\sigma\ge1$, the same argument also shows 
\begin{align*}
&\left|{v_2}\partial_x v_{2\theta}(|v_{1\theta}|^{2\sigma-2}v_{1\theta}-|v_{2 \theta}|^{2\sigma-2}v_{2 \theta})\right|\\
&\lesssim |v_2|\left(|\partial_x v_2|+|\partial_x v_{\mathrm p}|\right)\left\{|v_1-v_2|^{2\sigma-1}+(|v_1|^{2\sigma-2}+|v_2|^{2\sigma-2})|v_1-v_2|\right\}\\
&\lesssim \left(|v_1|^{2\sigma-1}+|v_2|^{2\sigma-1}\right)\left\{|\partial_x v_2|+(| \log t | |\varphi|^2+1)|\partial_x \varphi|\right\}|v_1-v_2|. 
\end{align*}
Hence, the same argument as above based on \eqref{G_proof_2} implies
\begin{align}
\nonumber
&\|\overline{v_2}\partial_x v_{2\theta}(|v_{1\theta}|^{2\sigma-2}v_{1\theta}-|v_{2 \theta}|^{2\sigma-2}v_{2 \theta})\|
+\|\overline{v_2\partial_x v_{2\theta}}(|v_{1\theta}|^{2\sigma-4}v_{1\theta}^3-|v_{2 \theta}|^{2\sigma-4}v_{2 \theta}^3)\|\\
\label{G_proof_16}
&\lesssim t^{2\beta+\alpha/4}\<M\>^{2\sigma}\>\|v_1-v_2\|_{\mathcal X}.
\end{align}
By \eqref{G_proof_12}--\eqref{G_proof_16}, 
we obtain
\begin{align*}
\|\partial_x  I_\sigma[v_1,v_2]\|\lesssim t^{2\beta+\alpha/4}\<M\>^{2\sigma}\|v_1-v_2\|_{\mathcal X},
\end{align*}
which, together with \eqref{G_proof_11}, shows the desired bound \eqref{proposition_G_2}. This completes the proof. 
\end{proof}

\begin{remark}
\label{remark_G}
For the proof of \eqref{G_proof_13}--\eqref{G_proof_16}, the condition $\sigma\ge1$ is crucial. When $1/2<\sigma<1$, we only have a similar estimate as \eqref{proposition_G_2} with $\|\vec v_1-\vec v_2\|_{\X}$ replaced by $\|\vec v_1-\vec v_2\|_{\X}^{2\sigma-1}$. 
\end{remark}

Next, for the error terms $e_1$ and $e_2$, we have: 

\begin{proposition}
\label{proposition_e_1}
Let $\alpha>0$, $\nu>0$ and $\delta\le1$. Then, for any $\varphi\in H^{\max\{2\delta,1\}}$ and $0<t\le 1$, 
\begin{align*}
\sqrt{Q_\alpha[e_1](t)}\lesssim t^{\delta-\max\{1/2,\alpha/2\}-\nu},\quad 
\sqrt{Q_\alpha[ie_2](t)}\lesssim |\lambda_1| t^{\delta-\max\{1/2,\alpha/2\}-1-\nu}
\end{align*}
\end{proposition}

In order to prove this proposition, we first prepare a few lemmas: 

\begin{lemma}
\label{lemma_R}
For any $s\in \R$, $0\le \delta\le1$ and $t\ge0$, 
$\|\mathcal R(t)f\|_{H^s}\lesssim t^\delta \|f\|_{H^{2\delta+s}}
$. 
\end{lemma}

\begin{proof}
Since  $\mathcal R(t)=e^{-itH_0}-I=\F^{-1}(e^{-it|\xi|^2/2}-1)\F$, the assertion follows from the bound $|e^{-it|\xi|^2/2}-1|\lesssim (t|\xi|^2)^\delta$ for $0\le \delta\le1$
\end{proof}

\begin{lemma}
\label{lemma_v_p}
Let $\gamma>0$, $0< s\le1$, $\nu>0$, $\varphi_1\in H^{\max\{s,1/2+\nu\}}(\R)$ and $\varphi_2\in H^s(\R)$. Then 
$$
\|e^{i\gamma \varphi_1}\varphi_2\|_{H^s}\lesssim \<\gamma\>\left(1+\|\varphi_1\|_{H^{\max\{s,1/2+\nu\}}}\right)\|\varphi_2\|_{H^{s}},
$$
where $\|\varphi_2\|_{H^{s}}$ is replaced by $\|\varphi_2\|_{H^{s+\nu}}$ when $s=1/2$. 
\end{lemma}

\begin{proof}
Assume $\gamma=1$ without loss of generality. The case $s=1$ follows by the embedding $L^\infty(\R)\subset H^1(\R)$. 
For $0<s<1$ we use the norm equivalence $\|f\|_{\dot H^s}\sim \|f\|_{\dot B_{2,2}^s}$ for $f\in H^s$ (see \cite{BeLo}), where
$$
\|f\|_{\dot B_{p,q}^\alpha}=\left(\int_0^\infty t^{-1-qs}\sup_{|y|\le t}\|\tau_yf-f\|_{L^p}^qdt\right)^{1/q} 
$$
and $\tau_yf(x)=f(x-y)$. Since 
\begin{align*}
\left|\tau_y[e^{i\varphi_1}\varphi_2]-e^{i\varphi_1}\varphi_2\right|
&\le \left|(e^{i\tau_y\varphi_1}-e^{i\varphi_1})\varphi_2\right|+|e^{i\tau_y\varphi_1}|\left|\tau_y\varphi_2-\varphi_2\right|\\
&\le \left|\tau_y\varphi_1-\varphi_1\right||\varphi_2|+\left|\tau_y\varphi_2-\varphi_2\right|,
\end{align*}
it holds for $1/2<s<1$ that
\begin{align*}
\|e^{i\varphi_1}\varphi_2\|_{\dot H^s}
\lesssim \|\varphi_1\|_{\dot H^s}\|\varphi_2\|_{L^\infty}+\|\varphi_2\|_{\dot H^s}
\lesssim \left(1+\|\varphi_1\|_{ H^s}\right)\|\varphi_2\|_{ H^s}.
\end{align*}
For $0<s<1/2$, 
using H\"older's inequality 
$
\|(\tau_y\varphi_1-\varphi_1)\varphi_2\|\le \|\tau_y\varphi_1-\varphi_1\|_{L^{\frac1s}}\|\varphi_2\|_{L^{\frac{2}{1-2s}}}
$ and the embeddings $ H^s\subset L^{\frac{2}{1-2s}}$ and $H^{1/2+\nu}\subset \dot B_{1/s,2}^s$ for $\nu>0$, we similarly have
\begin{align*}
\|e^{i\varphi_1}\varphi_2\|_{\dot H^s}
\lesssim \|\varphi_1\|_{\dot B_{1/s,2}^s}\|\varphi_2\|_{\dot H^s}+\|\varphi_2\|_{\dot H^s}
\lesssim \left(1+\|\varphi_1\|_{H^{1/2+\nu}}\right)\|\varphi_2\|_{ H^s}. 
\end{align*}
Finally, if $s=1/2$, then
$$
\|e^{i\varphi_1}\varphi_2\|_{\dot H^{1/2}}\lesssim \|\varphi_1\|_{\dot H^{1/2}}\|\varphi_2\|_{L^\infty}+\|\varphi_2\|_{\dot H^{1/2}}\lesssim (1+\|\varphi_1\|_{H^{1/2+\nu}})\|\varphi_2\|_{H^{1/2+\nu}}.
$$
These three estimates imply the desired result. 
\end{proof}

\begin{proof}[Proof of Proposition \ref{proposition_e_1}]
In what follows,  $C_s$ denotes constants depending on $\|\varphi\|_{H^s}$, but not on $t$, which may vary line to line. Set $\gamma_1=-\lambda_1\log |t|$, $\gamma_2=-\lambda_2t^{\sigma-1}/(\sigma-1)$, $\varphi_\sigma=e^{-i\gamma_2|\varphi|^{2\sigma}}\varphi$ and $\varphi_1=|\varphi|^{2}$. Then
\begin{align*}
v_{\mathrm p}=e^{i\gamma_1|\varphi|^2+i\gamma_2|\varphi|^{2\sigma}}\overline \varphi=e^{i\gamma_1\varphi_1}\overline{\varphi_\sigma},\quad
\partial_x v_{\mathrm p}=e^{i\gamma_1\varphi_1}(i\gamma_1\overline{\varphi_\sigma}\nabla\varphi_1+\nabla \overline{\varphi_\sigma}).
\end{align*}
Let $0\le \delta\le1$ and $\nu>0$. Since $|\gamma_2|\lesssim1$ on $(0,1]$ for $\sigma>1$, Lemmas \ref{lemma_R} and \ref{lemma_v_p} with $\varphi_2=\overline{\varphi_\sigma}$ for $\delta_2<1/2$ and $\varphi_2=(i\gamma_1\overline{\varphi_\sigma}\nabla\varphi_1+\nabla \overline{\varphi_\sigma})$ for $1/2\le \delta_2\le1$ imply
\begin{align*}
\|e_1(t)\|\lesssim t^{\delta}\|v_{\mathrm p}(t)\|_{H^{2\delta}}\le C_{\max\{2\delta,1/2+\nu\}}t^{\delta-\nu}
\end{align*} 
Similarly, we also obtain for any $0\le \delta_1\le 1/2$
\begin{align*}
\|\partial_xe_1(t)\|\lesssim t^{\delta_1}\|\partial_x v_{\mathrm p}(t)\|_{H^{2\delta_1}}\le C_{2\delta_1+1}t^{\delta_1-\nu}.
\end{align*} 
These two estimates imply\begin{align*}
\sqrt{Q_\alpha[e_1](t)}
\lesssim \|\partial_x  e_1(t)\|+t^{-\alpha/2}\|e_1(t)\|+t^{-1/2}\|e_1(t)\|
\le C_{\max\{2\delta,1\}}t^{\delta-\max\{1/2,\alpha/2\}-\nu}
\end{align*}
for any $1/2\le \delta\le1$, where we chose $\delta_1=\delta-1/2$.  
Recalling the formula
$$
e_2
=\sum_{j=1}^2\lambda_j t^{\sigma_j-2}\left\{-\mathcal R|v_\mathrm p|^{2\sigma_j}v_\mathrm p+(\sigma_j+1)|\varphi|^{2\sigma_j}e_1+ \sigma_j |\varphi|^{2\sigma_j-2}v_{\mathrm p},^2\overline{e_1}\right\}$$
we obtain the desired bound for $e_2$ by the same argument based on Lemmas \ref{lemma_R} and \ref{lemma_v_p}. 
\end{proof}

With Propositions \ref{proposition_G} and \ref{proposition_e_1} at hand, we are ready to complete the proof of Theorem \ref{theorem_1}.

\begin{proof}[Proof of Theorem \ref{theorem_1}]
Suppose $2/3\le \alpha<1$ and $\vec v_*,\vec v_1,\vec v_2\in \mathcal X(T,\alpha,\beta,M)$. Then \eqref{proposition_G_1} implies
\begin{align}
\sup_{0<t\le T}t^{-\beta}\int_0^t \sqrt{Q_\alpha[iG[v_*]](s)}ds
\label{G_1}
\lesssim T^{\beta+3\alpha /4-1/2}\<M\>^{2\sigma+1}\lesssim T^\beta \<M\>^{2\sigma+1}. 
\end{align}
Moreover, it follows from Proposition \ref{proposition_e_1} that
\begin{align*}
\sup_{0<t\le T} t^{-\beta} \sqrt{Q_\alpha[e_1](t)}+\sup_{0<t\le T}t^{-\beta} \int_0^t \sqrt{Q_\alpha[ie_2](s)}ds\lesssim T^{-\beta+\delta-1/2-\nu}
\end{align*}
as long as $1/2<\delta\le1$, $\nu>0$ and $\varphi\in H^{2\delta}$. Thus, if $\varphi\in H^{1+\ep}$ with some $\ep>0$, $0<\beta<\min\{\ep/2,1/2\}$ and $0<\nu<\ep/2-\beta$, then
\begin{align}
\|\Phi[\vec v_*]\|_{\mathcal X}
\lesssim T^{\delta_0}\<M\>^{2\sigma+1}
\label{Theorem_1_proof_1}
\end{align}
with $\delta_0=\min\{\beta,\ep/2-\beta-\nu\}$. 
Since the error terms $\vec e_1,\vec e_2$ do not appear in the difference $\Phi[\vec v_1]-\Phi[\vec v_2]$,  the same argument based on the estimate \eqref{proposition_G_2} shows
\begin{align}
\|\Phi[\vec v_1]-\Phi[\vec v_2]\|_{\mathcal X}
\label{Theorem_1_proof_2}
\lesssim T^{\beta}\<M\>^{2\sigma}\|\vec v_1-\vec v_2\|_{\mathcal X}.
\end{align}
Therefore, for any $M>0$ there exists $T_M>0$ such that $\Phi$ is a contraction on $\mathcal X(T,\alpha,\beta,M)$ for any $0<T\le T_M$ and there exists a unique solution $\vec v_*=(v_*,\overline{v_*})^{\mathrm T}\in C((0,T];\mathcal H^1(\R))$ to \eqref{integral_equation} satisfying the asymptotic condition
\begin{align}
\label{Theorem_1_proof_3}
\|\partial_x v_*(t)\|+t^{-\alpha/2}\|v_*(t)\|\lesssim t^\beta,\quad t\to +0.
\end{align}
By Lemma \ref{lemma_propagator}, $v_*$ satisfies \eqref{equation_v_1} in $H^{-1}$ which, together with \eqref{v_p}, shows that $v:=v_*+v_{\mathrm p}$ solves \eqref{v} in $H^{-1}$ and the following Duhamel formula in $H^1$: 
\begin{align}
\label{Duhamel}
v(t)=e^{-i(t-T)H_0}v(T)-i \int_{T}^t e^{-i(t-s)H_0}\left(\lambda_1 s^{-1}|v|^2v+\lambda_2 s^{\sigma-2}|v|^{2\sigma}v\right)ds,\quad 0<t\le T.
\end{align}
Conversely, for any solution $v\in C((0,T];H^1(\R))$ to \eqref{Duhamel} with a given initial datum $v(T)\in H^1$ satisfying \eqref{Theorem_1_proof_3} with some $2/3\le \alpha<1$ and $0<\beta<\min\{\ep/2,1/2\}$, $\vec v_*:=\vec v-\vec v_{\mathrm p}$ solves \eqref{equation_v_1} in $H^{-1}$ and \eqref{integral_equation} in $H^1$. 

Next, we prove the uniqueness of $v$. Let $v_j\in C((0,T_0];H^1(\R))$ for $j=1,2$ be two solutions to \eqref{v} satisfying \eqref{Theorem_1_proof_3} with some $2/3\le\alpha_1,\alpha_2<1$ and $\beta_1,\beta_2>0$, respectively. Let $\alpha_0=\min\{\alpha_1,\alpha_2\}$ and $\beta_0=\min\{\beta_1,\beta_2\}$. Then there exists $M_0>0$ such that,  for any $0<T\le T_0$, 
$$
\|\partial_xv_j(t)-\partial_xv_{\mathrm p}(t)\|+t^{-\alpha_0/2}\|v_j(t)-v_{\mathrm p}(t)\|\le M_0 t^{\beta_0},\quad 0<t\le T. 
$$
Moreover, by the same argument as above, $v_j$ satisfy \eqref{Duhamel} and hence
$$
\vec v_1(t)-\vec v_2(t)=-\int_0^t \mathcal U(t,s)\left\{\vv{ (i G)}[v_1](s)-\vv{( iG)}[v_2](s)\right\}ds,\quad 0<t\le T,
$$
where $\vec v_j=(v_j,\overline{v_j})^{\mathrm T}$. The same argument as that for showing \eqref{Theorem_1_proof_2} then implies
$$
\|\vec v_1-\vec v_2\|_{\mathcal X} \lesssim T^{\beta_0} \<M_0\>^{2\sigma}\|\vec v_1-\vec v_2\|_{\mathcal X}. 
$$
This shows $v_1(t)=v_2(t)$ for $0<t\le T$ with sufficiently small $T$ and hence $v_1\equiv v_2$ by the well-posedness of the Cauchy problem for \eqref{v} in $(0,T_0]$ (see \cite[Theorem 4.11.1]{Cazenave}). Therefore, the above $v\in C((0,T];H^1(\R))$ is a unique solution to \eqref{v} satisfying \eqref{Theorem_1_proof_3}. Note that since the above equation for $\vec v_1(t)-\vec v_2(t)$ does not have error terms $\vec e_1,\vec e_2$, this argument for showing the uniqueness works well by assuming $\varphi \in H^{1}(\R)$ only. 

Finally, we translate these results to the original NLS \eqref{NLS}. Let $u$ be the inverse pseudo-conformal transform of $v$ defined by
$
u=\mathcal M(t)\mathcal D(t)\mathcal T^{-1}\overline v
$. By the above properties for $v$ and the equivalence of \eqref{theorem_1_1} and \eqref{v-v_p}, $u\in C([T^{-1},\infty);L^2(\R))$ is a unique solution to \eqref{NLS} satisfying $e^{itH_0}u\in C([T^{-1},\infty);\F H^1(\R))$ and \eqref{theorem_1_1}. 
By \eqref{w_p-widetilde_w_p_1} and \eqref{w_p-widetilde_w_p_2}, $u$ also satisfies 
\eqref{theorem_1_2} if in addition $\alpha/2<\sigma-1$. Since the Cauchy problem for \eqref{NLS} is globally well-posed in $L^2(\R)$ if $0<\sigma<2$ (\cite{Tsutsumi_1987}), $u$ can be extended uniquely backward in time from time $t=T^{-1}$, satisfying $u\in C(\R;L^2(\R))$. Moreover, we have $e^{itH_0}u\in C(\R;\F H^1(\R))$ thanks to $e^{iT^{-1}H_0}u(T^{-1})\in \F H^1(\R)$  and the persistence of the $\mathcal F H^1$-regularity for $0<\sigma<2$ (see e.g. \cite[Proposition 2.2]{Masaki_2017} where a simple proof for the case $\lambda_1<0$ and $\lambda_2=0$ can be found and the same proof also works well for $\lambda_1>0$ and $\lambda_2\neq0$). The modified wave operator 
$
W_+:\mathcal F H^{1+\ep}(\R)\ni \F^{-1} \varphi\mapsto u(0)\in \mathcal F H^1(\R)
$ thus  
is well-defined. This completes the proof. 
\end{proof}

\section{Proof of Theorem \ref{theorem_2}}
\label{section_proof_Theorem_2}
The basic strategy of the proof of Theorem \ref{theorem_2} is almost the same as that of Theorem \ref{theorem_1}, namely we want to construct the solution $\vec v_*$ to \eqref{integral_equation}. However, we used the condition $\lambda_1>0$ in an essential way to prove Lemma \ref{lemma_energy_estimate}. Instead, we use the following

\begin{lemma}
\label{lemma_Theorem_2_1}
Let $\varphi\in H^1(\R)$ and $\max\{1,2|\lambda_1|\|\varphi\|_{L^\infty(\R)}^2\}<\alpha<2$. Then, there exists $t_0>0$ such that for all $\psi_0\in H^1(\R)$ and $0<s\le t\le t_0$
$$
Q_\alpha[\mathcal U(t,s)\vec \psi_0](t)\lesssim Q_\alpha[\psi_0](s).
$$
\end{lemma}

\begin{proof}
Thanks to Lemma \ref{lemma_energy_identity}, it is enough to show 
$
Q_\alpha[\psi](t)\lesssim Q_\alpha[\psi](s)
$ for the solution $\psi\in C((0,\infty);H^1(\R))$ to \eqref{equation_psi_1}. Note that both 
$
Q_\alpha[\psi]
$ and $\widetilde Q_\alpha[\psi]$ are comparable to $\frac14\|\partial_x \psi\|^2+t^{-\alpha}\|\psi\|^2$ for sufficiently small $t>0$ under the above assumption. Set $\Lambda=2|\lambda_1|\|\varphi\|_{L^\infty(\R)}^2$ for short. The identity \eqref{lemma_energy_identity_2} and the same argument as in the proof of Lemma \ref{lemma_energy_estimate} show that there exist $C,\ep>0$ with $\alpha-\Lambda-\ep>0$ such that 
\begin{align*}
\frac{d}{dt}\widetilde Q_\alpha[\psi]
&\le -(\alpha-\Lambda-\ep)t^{-\alpha-1}\|\psi\|^2+(\Lambda+\ep) t^{-2}\|\psi\|^2\\
&\quad 
+t^{-1}|\lambda_1|\|\varphi\|_{L^\infty}\|\partial_x \varphi\|\|\partial_x \psi\|\|\psi\|_{L^\infty}
+t^{\sigma-2}|\lambda_2|\|\varphi\|_{L^\infty}^{2\sigma-1}\|\partial_x \varphi\|\|\partial_x \psi\|\|\psi\|_{L^\infty}
\\
&\le -\frac{(\alpha-\Lambda-\ep)t^{-\alpha-1}\|\psi\|^2}{2}+(\Lambda+\ep) t^{-2}\|\psi\|^2+\frac{Ct^{\alpha/3-1}\|\partial_x\psi\|^2}{4}\\
&\le Ct^{\alpha/3-1}\widetilde Q_\alpha[\psi]
\end{align*}
for all $0<t\le t_0$ with sufficiently small $t_0>0$ so that $(\alpha-\Lambda-\ep)t_0^{-\alpha+1}>2(\Lambda+\ep)$. This implies the desired estimate. \end{proof}

\begin{proof}[Proof of Theorem \ref{theorem_2}]
This lemma and Proposition \ref{proposition_G} imply that \eqref{G_1} and \eqref{Theorem_1_proof_2} still hold for $0<t\le t_0$ in the present setting. Moreover, it follows from Proposition \ref{proposition_e_1} and this lemma that, for any $0<T\le t_0$, 
\begin{align*}
\sup_{0<t\le T} t^{-\beta} \sqrt{Q_\alpha[e_1](t)}+\sup_{0<t\le T}t^{-\beta} \int_0^t \sqrt{Q_\alpha[ie_2](s)}ds\lesssim T^{-\beta+\delta-\alpha/2-\nu}. 
\end{align*}
Since $\alpha<2$ and $0<\beta<1-\alpha/2$, one can choose $0\le \delta\le1$ and $\nu>0$ so that $-\beta+\delta-\alpha/2-\nu>0$. The remaining part of the proof is completely the same as that for Theorem \ref{theorem_1}.  
\end{proof}

\section*{Acknowledgments}
M. K. is partially supported by JSPS KAKENHI Grant Number JP24K06796. H. M. is partially supported by JSPS KAKENHI Grant Numbers JP21K03325 and JP24K00529. This work was supported by the Research Institute for Mathematical Sciences, an International Joint Usage/Research Center located in Kyoto University.


\end{document}